\documentclass[a4paper,11pt,reqno]{amsart}

\usepackage{lmodern}
\usepackage[utf8]{inputenc}
\usepackage[T1]{fontenc}

\usepackage[english]{babel}

\usepackage{amsmath,amsfonts,amssymb,amsthm}
\usepackage{mathtools}

\numberwithin{equation}{section}

\usepackage{tikz}

\usepackage{hyperref}
\hypersetup{
  pdftitle={From plactic monoids to hypoplactic monoids},
  pdfauthor={Ricardo P. Guilherme},
  pdfsubject={20M10; 05E16, 20M05},
  pdfkeywords={quasi-crystal, hypoplactic monoid, plactic monoid, equidivisible}
}

\theoremstyle{plain}
\newtheorem{thm}{Theorem}[section]

\newtheorem{lem}[thm]{Lemma}
\newtheorem{prop}[thm]{Proposition}

\theoremstyle{definition}

\newtheorem{dfn}[thm]{Definition}
\newtheorem{exa}[thm]{Example}

\theoremstyle{remark}
\newtheorem{rmk}[thm]{Remark}

\newcommand*{\dtgterm}[1]{\emph{#1}}  

\newcommand*{\avoidrefbreak}{\nolinebreak[3] }   
\newcommand*{\avoidcitebreak}{\nolinebreak[3] }  

\newcommand*{\itmform}[1]{\textup{(#1)}}
\newcommand*{\itmref}[1]{\itmform{\ref{#1}}}

\newcommand*{\comboref}[2]{%
  \ifdefined\hyperref%
    \hyperref[#2]{#1\avoidrefbreak \textup{\ref*{#2}}}%
  \else%
    #1\avoidrefbreak \textup{\ref{#2}}%
  \fi%
}

\newcommand*{\Z}{\mathbb{Z}}               
\newcommand*{\R}{\mathbb{R}}               

\DeclarePairedDelimiter{\parens}{\lparen}{\rparen}

\providecommand*{\given}{\relax}
\DeclarePairedDelimiterX{\set}[1]{\{}{\}}{%
  \renewcommand*{\given}{\mathclose{}\mathrel{\setsymbol[\delimsize]}\mathopen{}}%
  #1%
}
\DeclarePairedDelimiterXPP{\setcardin}[1]{\#}{\lvert}{\rvert}{}{#1}

\newcommand*{\powersetsymbol}{\mathcal{P}}
\DeclarePairedDelimiterXPP{\powerset}[1]{\powersetsymbol}{\lparen}{\rparen}{}{#1}

\DeclarePairedDelimiterX{\gnrt}[1]{\langle}{\rangle}{%
  \renewcommand*{\given}{\mathclose{}\mathrel{\setsymbol[\delimsize]}\mathopen{}}%
  #1%
}

\newcommand*{\pressymbol}[1][]{#1\vert}
\DeclarePairedDelimiterX{\pres}[1]{\langle}{\rangle}{%
  \renewcommand*{\given}{\mathclose{}\mathrel{\pressymbol[\delimsize]}\mathopen{}}%
  #1%
}

\newcommand*{\ew}{\epsilon}                     

\newcommand*{\lbedge}[1]{\xrightarrow[\hphantom{n-g}]{#1}}  

\DeclarePairedDelimiterX{\innerp}[2]{\langle}{\rangle}{#1,#2}  
\newcommand*{\vc}[1]{\mathbf{#1}}  

\newcommand*{\wbar}[1]{\overline{#1}}  

\newcommand*{\glin}{\mathfrak{gl}}  
\newcommand*{\symp}{\mathfrak{sp}}  

\newcommand*{\alphav}{\alpha^{\vee}}
\newcommand*{\betav}{\beta^{\vee}}

\newcommand*{\tA}{\mathfrak{A}}
\newcommand*{\tB}{\mathfrak{B}}
\newcommand*{\tC}{\mathfrak{C}}
\newcommand*{\tD}{\mathfrak{D}}

\newcommand*{\tAn}{\tA_{n}}
\newcommand*{\tBn}{\tB_n}
\newcommand*{\tCn}{\tC_n}
\newcommand*{\tDn}{\tD_n}

\newcommand*{\crst}[1]{\mathcal{#1}}

\DeclareMathOperator{\wt}{wt}
\newcommand*{\Koe}{\tilde{e}}
\newcommand*{\Kof}{\tilde{f}}
\newcommand*{\Koec}{\tilde{\varepsilon}}
\newcommand*{\Kofc}{\tilde{\varphi}}

\newcommand*{\undf}{\bot}  

\DeclareMathOperator{\sgn}{sgn}       
\newcommand*{\tpsgn}{\sgn^{\otimes}}  

\newcommand*{\crtA}{\crst{A}}

\newcommand*{\crtC}{\crst{C}}

\newcommand*{\crtAn}{\crtA_n}

\newcommand*{\crtCn}{\crtC_n}

\newcommand*{\ftpqcm}[1]{#1^{\otimes}}    

\newcommand*{\tpqKoe}{\qKoe^{\otimes}}
\newcommand*{\tpqKof}{\qKof^{\otimes}}
\newcommand*{\tpqKoec}{\qKoec^{\otimes}}
\newcommand*{\tpqKofc}{\qKofc^{\otimes}}

\newcommand*{\qcrst}[1]{\mathcal{#1}}
\newcommand*{\qcrstQ}{\qcrst{Q}}

\newcommand*{\qKoe}{\ddot{e}}
\newcommand*{\qKof}{\ddot{f}}
\newcommand*{\qKoec}{\ddot{\varepsilon}}
\newcommand*{\qKofc}{\ddot{\varphi}}

\newcommand*{\dotimes}{\mathbin{\ddot{\otimes}}}  

\newcommand*{\qtpsgn}{\sgn^{\dotimes}}  

\newcommand*{\qctC}{\qcrst{C}}

\newcommand*{\qcmon}[1]{\qcrst{#1}}
\newcommand*{\qcrstM}{\qcmon{M}}

\newcommand*{\fqtpqcm}[1]{#1^{\dotimes}}      

\newcommand*{\qtpqKoe}{\qKoe^{\dotimes}}
\newcommand*{\qtpqKof}{\qKof^{\dotimes}}
\newcommand*{\qtpqKoec}{\qKoec^{\dotimes}}
\newcommand*{\qtpqKofc}{\qKofc^{\dotimes}}

\DeclareMathOperator{\plac}{plac}
\newcommand*{\plco}{\approx}

\DeclareMathOperator{\hypo}{hypo}
\newcommand*{\hyco}{\mathrel{\ddot{\sim}}}
\newcommand*{\nhyco}{\mathrel{\not\hyco}}

\newlength\basiccrystalx
\newlength\basiccrystaly
\newlength\widecrystalx

\makeatletter
\if 0\@ptsize
\else%
  \if 1\@ptsize
  \else%
    \if 2\@ptsize
    \else%
    \fi%
  \fi%
\fi
\makeatother

\tikzset{%
  basiccrystal/.style={%
  x=\basiccrystalx,%
  y=\basiccrystaly,%
    every node/.append style={%
      execute at begin node=$,%
      execute at end node=$,%
    },%
    every path/.append style={%
      ->,%
      auto,%
      every node/.append style={%
        execute at begin node=\scriptstyle,%
      },%
    },%
  },%
  widecrystal/.style={%
  x=\widecrystalx,%
  y=\basiccrystaly,%
    every node/.append style={%
      execute at begin node=$,%
      execute at end node=$,%
    },%
    every path/.append style={%
      ->,%
      auto,%
      every node/.append style={%
        execute at begin node=\scriptstyle,%
      },%
    },%
  },%
}

\begin{document}

\title{From plactic monoids to hypoplactic monoids}

\author{Ricardo P. Guilherme}
\address{%
Center for Mathematics and Applications (NOVA Math)\\
NOVA FCT
}
\email{rj.guilherme@campus.fct.unl.pt}
\thanks{The author would like to thank the organizers of the Conference on Theoretical and Computational Algebra 2023 in Pocinho, Portugal, for providing such an inspiring scientific event.}
\thanks{This work was funded by national funds through the FCT -- Funda\c{c}\~{a}o para a Ci\^{e}ncia e a Tecnologia, I.P., under the scope of the projects UIDB/00297/2020 and UIDP/00297/2020 (Center for Mathematics and Applications).}

\begin{abstract}
The plactic monoids can be obtained from the tensor product of crystals.
Similarly, the hypoplactic monoids can be obtained from the quasi-tensor product of quasi-crystals.
In this paper, we present a unified approach to these constructions by expressing them in the context of quasi-crystals.
We provide a sufficient condition to obtain a quasi-crystal monoid for the quasi-tensor product from a quasi-crystal monoid for the tensor product.
We also establish a sufficient condition for a hypoplactic monoid to be a quotient of the plactic monoid associated to the same seminormal quasi-crystal.
\end{abstract}



\maketitle


\section{Introduction}
\label{sec:introduction}

The plactic monoid, formally introduced  by Lascoux and Sch\"{u}tzenberger\avoidcitebreak \cite{LS81}, is an algebraic object of great interest, with connections to several fields such as representation theory, combinatorics\avoidcitebreak \cite{Ful97}, symmetric functions, and Schubert polynomials\avoidcitebreak \cite{LS85,LS89}.
It was also used to give a first rigorous proof of the Littlewood--Richardson rule\avoidcitebreak \cite{LR34}.
It originally emerged from Young tableaux and the Schensted insertion algorithm\avoidcitebreak \cite{Sch61} with a presentation given by the Knuth relations\avoidcitebreak \cite{Knu70}.

Kashiwara\avoidcitebreak \cite{Kas90,Kas91,Kas94}, following the work by Date, Jimbo and Miwa\avoidcitebreak \cite{DJM90}, introduced crystal bases for modules of quantized universal enveloping algebras (also known as quantum groups), discovered independently by Drinfel'd\avoidcitebreak \cite{Dri85} and Jimbo\avoidcitebreak \cite{Jim85}.
Crystal bases can be described by weighted labelled graphs that are called crystal graphs.
Kashiwara showed that the plactic monoid arises from the crystal basis associated with the vector representation of the quantized universal enveloping general linear Lie algebra by identifying elements in the same position of isomorphic connected components of the associated crystal graph.
This result allowed a deeper study of the plactic monoid and its generalization, because the underlying construction still results in a monoid for crystal bases associated with other quantized universal enveloping algebras, as it only relies on the definition of seminormal crystals and their tensor product. 
Thus, based on the work by Kashiwara and Nakashima\avoidcitebreak \cite{KN94}, Lecouvey\avoidcitebreak \cite{Lec02,Lec03} presented comprehensive descriptions of the plactic monoids for the Cartan types $\tBn$, $\tCn$, and $\tDn$, which later appeared in a survey\avoidcitebreak \cite{Lec07}.
In recent works, Cain, Gray and Malheiro\avoidcitebreak \cite{CGM15f,CGM19} presented rewriting systems and biautomatic structures for these monoids.

The hypoplactic monoid, introduced by Krob and Thibon\avoidcitebreak \cite{KT97} and studied by Novelli\avoidcitebreak \cite{Nov00}, emerged from a noncommutative realization of quasi-symmetric functions analogous to the realization of symmetric functions by the plactic monoid presented by Lascoux and Sch\"{u}tzenberger\avoidcitebreak \cite{LS81}.
It was originally obtained from quasi-ribbon tableaux and an insertion algorithm with a presentation consisting of the Knuth relations and the quartic relations.
A comparative study with other monoids was done by Cain, Gray and Malheiro in\avoidcitebreak \cite{CGM15r}, where a rewriting system and a biautomatic structure for the hypoplactic monoid is presented.
Recently, Cain, Malheiro and Ribeiro\avoidcitebreak \cite{CMR22} provide a complete description of the identities satisfied by the hypoplactic monoid.

To obtain a construction of the hypoplactic monoid analogous to the construction of the plactic monoid from crystals, 
a first notion of quasi-crystal graph was introduced by Krob and Thibon\avoidcitebreak \cite{KT99}.
To overcome some limitations of this notion, 
Cain and Malheiro\avoidcitebreak \cite{CM17crysthypo} described a new notion of quasi-crystal graph, 
which is equivalent to another notion considered recently by Maas-Gari\'{e}py\avoidcitebreak \cite{Maa23},
from which the hypoplactic monoid arises by identifying words in the same position of isomorphic connected components.
These notions of quasi-crystal graphs are based on the crystal graph for Cartan type $\tAn$, and the construction does not result in a monoid if this crystal graph is replaced by the crystal graph for another Cartan type.

Cain, Malheiro, and the present author\avoidcitebreak \cite{CGM23quasi-crystals} introduced a general notion of quasi-crystals associated to a root system.
This notion is further studied in\avoidcitebreak \cite{CMRR23}.
It allows the construction of the hypoplactic monoid by identifying words in the same position of isomorphic connected components of a quasi-crystal associated to type $\tAn$.
Moreover, it allows the generalization of the hypoplactic monoid, because this construction still results in a monoid for any other seminormal quasi-crystal, as it only relies on the definition of quasi-tensor product of seminormal quasi-crystals, also introduced in\avoidcitebreak \cite{CGM23quasi-crystals}.

For this notion, crystals are quasi-crystals; in particular, the class of seminormal crystals is contained on the class of seminormal quasi-crystals.
Also, the definition of tensor product of seminormal crystals extends in a natural way to seminormal quasi-crystals.
Thus, quasi-crystals form a framework where the constructions of plactic monoids and hypoplactic monoids can be expressed.

The aim of this paper is to formalize and develop this unified approach where plactic and hypoplactic monoids are associated with the same combinatorial objects: quasi-crystals.
This contrasts with previous approaches in the literature, which obtain independently plactic monoids from crystals and hypoplactic monoids from (some notion of) quasi-crystals.
As shown by this unified construction, plactic and hypoplactic monoids emerge from tensor and quasi-tensor products of quasi-crystals, respectively.
As the underlying combinatorial objects are the same, it allows a comparative study between plactic and hypoplactic monoids, where some well-known constructions for the classical case are shown to be consequence of properties of the underlying quasi-crystals.
This unified construction also allows a deeper understanding of the results obtained in\avoidcitebreak \cite{CGM23quasi-crystals} for the relation between the plactic and hypoplactic monoids of type $\tCn$.

This paper is structured as follows.
\comboref{Section}{sec:preliminaries} introduces the necessary background on root systems and quasi-crystals.
\comboref{Section}{sec:plmqc} presents a construction of the plactic monoid based on the tensor product of quasi-crystals, 
and \comboref{Section}{sec:hypomqc} presents an analogous construction of the hypoplactic monoid based on the quasi-tensor product of quasi-crystals.
From these sections, plactic and hypoplactic monoids are obtained as a quotient of the free quasi-crystal monoid for the respective product of quasi-crystals.
\comboref{Section}{sec:tpqcmqtpqcm} shows a construction of quasi-crystal monoids for the quasi-tensor product from quasi-crystal monoids for the tensor product, whenever the monoids are equidivisible.
When applied to the free quasi-crystal monoids, it results in the constructions described in\avoidcitebreak \cite{CM17crysthypo,CGM23quasi-crystals}.
In \comboref{Section}{sec:hypomplacm}, it is given a sufficient condition for the hypoplactic monoid to be a quotient of the plactic monoid associated to the same quasi-crystal.
For type $\tAn$, this result is an alternative prove of the well-known fact that the classical hypoplactic monoid is a quotient of the classical plactic monoid, and for type $\tCn$, it is a step towards understanding the relation between plactic and hypoplactic monoids for this type.

\section{Preliminaries}
\label{sec:preliminaries}

In this section, we present the necessary background on crystals and quasi-crystals.
We will introduce crystals as a subclass of quasi-crystals, so we will not need to present a complete introduction to crystals.
We refer
to\avoidcitebreak \cite{Kas95} for an introduction to crystals as they originally emerged in connection to quantized universal enveloping algebras (also called quantum groups)
or\avoidcitebreak \cite{HK02} for a comprehensive background on this approch,
and to\avoidcitebreak \cite{BS17} for a study of crystals detached from their origin,
For a detailed study of quasi-crystals, see\avoidcitebreak \cite{CGM23quasi-crystals}.

\subsection{Root systems}
\label{subsec:rootsystems}

We first give the essential background on root systems, as these algebraic structures will be used to define crystals and quasi-crystals.
Root systems are commonly found in representation theory, in particular, they arise on the study of Lie groups and Lie algebras, but we will detach them from this context, as we will only introduce what we need for our purpose.
For further context see for example\avoidcitebreak \cite{FH91,EW06,Bum13}.

Let $V$ be a Euclidean space, that is, a real vector space with an inner product $\innerp{\,{\cdot}\,}{\,{\cdot}\,}$.
For $\alpha \in V$ other than $0$, denote by $r_{\alpha}$ the \dtgterm{reflection} in the hyperplane orthogonal to $\alpha$, which is given by
\[
r_{\alpha} (v) = v - \innerp[\big]{v}{\alphav} \alpha,
\quad \text{where} \quad
\alphav = \frac{2}{\innerp{\alpha}{\alpha}} \alpha,
\]
for each $v \in V$. Note that $r_\alpha$ is bijective, as $r_{\alpha} \parens[\big]{ r_{\alpha} (v) } = v$, for all $v \in V$. Also, $r_\alpha$ preserves the inner product, as $\innerp[\big]{r_\alpha (u)}{r_\alpha (v)} = \innerp{u}{v}$ for any $u, v \in V$.

\begin{dfn}
A \dtgterm{root system} in $V$ is a subset $\Phi$ of $V$ satisfying the following conditions:
\begin{enumerate}
\item
$\Phi$ is nonempty, finite, and $0 \notin \Phi$;

\item
$r_\alpha (\beta) \in \Phi$, for all $\alpha, \beta \in \Phi$;

\item
$\innerp[\big]{\alpha}{\betav} \in \Z$, for all $\alpha, \beta \in \Phi$;

\item
if $\alpha \in \Phi$ and $k \alpha \in \Phi$, then $k = \pm 1$.
\end{enumerate}
The elements of $\Phi$ are called \dtgterm{roots}, and the elements $\alphav$, with $\alpha \in \Phi$, are called \dtgterm{coroots}.
\end{dfn}


Together with a root system, we always fix an index set $I$ and \dtgterm{simple roots} $(\alpha_i)_{i \in I}$, that is, a collection of roots satisfying the following conditions:
\begin{itemize}
\item
$\set{\alpha_i \given i \in I}$ is a linearly independent subset of $V$; and

\item
every root $\beta \in \Phi$ can be expressed as $\beta = \sum_{i \in I} k_i \alpha_i$, where all $k_i$ are either nonnegative or nonpositive integers.
\end{itemize}


Finally, together with a root system we also consider the following structure.

\begin{dfn}
A \dtgterm{weight lattice} $\Lambda$ is a $\Z$-submodule of $V$ satisfying the following conditions:
\begin{enumerate}
\item
$\Lambda$ spans $V$;

\item
$\Phi \subseteq \Lambda$;

\item
$\innerp[\big]{\lambda}{\alphav} \in \Z$, for any $\lambda \in \Lambda$ and $\alpha \in \Phi$.
\end{enumerate}
The elements of $\Lambda$ are called \dtgterm{weights}.
\end{dfn}


In the subsequent sections, we always consider a root system $\Phi$ with weight lattice $\Lambda$ and index set $I$ for the simple roots $(\alpha_i)_{i \in I}$.
The only non-arbitrary root systems that will be considered are the root systems associated to Cartan types $\tAn$ and $\tCn$.

Let $n \geq 2$. Consider $V$ to be the real vector space $\R^n$ with the usual inner product,
and denote by $\vc{e_i} \in \R^n$ the $n$-tuple with $1$ in the $i$-th position, and $0$ elsewhere, $i = 1, 2, \ldots, n$.
The root system associated to Cartan type $\tAn$ based on the general linear Lie algebra $\glin_n$ consists of
$\Phi = \set{ \vc{e_i} - \vc{e_j} \given i \neq j }$,
the index set for the simple roots is $I = \set{1, 2, \ldots, n-1}$,
the simple roots are
$\alpha_i = \vc{e_i} - \vc{e_{i+1}}$, $i=1,2,\ldots,n-1$,
and the weight lattice is $\Lambda = \Z^{n}$.
For this type, we follow the notation used in\avoidcitebreak \cite{Lec02,CGM19}, although it is also commonly denoted in the literature by $\tA_{n-1}$.

The root system associated to Cartan type $\tCn$ based on the symplectic Lie algebra $\symp_{2n}$ consists of
$\Phi = \set{ \pm \vc{e_i} \pm \vc{e_j} \given i < j } \cup \set{ \pm 2\vc{e_i} \given i = 1, 2, \ldots, n }$,
the index set for the simple roots is $I = \set{1, 2, \ldots, n}$, the simple roots are
$\alpha_i = \vc{e_i} - \vc{e_{i+1}}$, $i=1,2,\ldots,n-1$, and $\alpha_n = 2\vc{e_n}$,
and the weight lattice is $\Lambda = \Z^{n}$.

For more examples of root systems see\avoidcitebreak \cite[Examples~2.4 to~2.10]{BS17}.

\subsection{Crystals and quasi-crystals}
\label{subsec:crystalsandquasi-crystals}

Consider $\Z \cup \set{ +\infty }$ to be the usual set of integers where we add a maximal element $+\infty$, that is, $m < +\infty$ for all $m \in \Z$.
Also, set $m + (+\infty) = (+\infty) + m = +\infty$, for all $m \in \Z$.

\begin{dfn}
\label{dfn:qc}
Let $\Phi$ be a root system with weight lattice $\Lambda$ and index set $I$ for the simple roots $(\alpha_i)_{i \in I}$.
A \dtgterm{seminormal quasi-crystal} $\qcrstQ$ of type $\Phi$ consists of a set $Q$ together with maps ${\wt} : Q \to \Lambda$, $\qKoe_i, \qKof_i : Q \to Q \sqcup \{\undf\}$ and $\qKoec_i, \qKofc_i : Q \to \Z \cup \set{ +\infty }$, for each $i \in I$, satisfying the following conditions:
\begin{enumerate}
\item\label{dfn:qcwt}
$\qKofc_i (x) = \qKoec_i (x) + \innerp[\big]{\wt (x)}{\alphav_i}$;

\item\label{dfn:qcqKoe}
if $\qKoe_i (x) \in Q$, then
$\wt \parens[\big]{ \qKoe_i (x) } = \wt (x) + \alpha_i$;

\item\label{dfn:qcqKof}
if $\qKof_i (x) \in Q$, then
$\wt \parens[\big]{ \qKof_i (x) } = \wt (x) - \alpha_i$;

\item\label{dfn:qciff}
$\qKoe_i (x) = y$ if and only if $x = \qKof_i (y)$;


\item\label{dfn:qcpinfty}
if $\qKoec_i (x) = +\infty$ then $\qKoe_i (x) = \qKof_i (x) = \undf$;

\item\label{dfn:qcsn}
if $\qKoec_i (x) \neq +\infty$, then
\[ \qKoec_i (x) = \max \set[\big]{ k \in \Z_{\geq 0} \given \qKoe_i^k (x) \in Q } \]
and
\[ \qKofc_i (x) = \max \set[\big]{ k \in \Z_{\geq 0} \given \qKof_i^k (x) \in Q }. \]
\end{enumerate}
for $x, y \in Q$ and $i \in I$.
The set $Q$ is called the \dtgterm{underlying set} of $\qcrstQ$, and the maps $\wt$, $\qKoe_i$, $\qKof_i$, $\qKoec_i$ and $\qKofc_i$ ($i \in I$) form the \dtgterm{quasi-crystal structure} of $\qcrstQ$.
Also, the map $\wt$ is called the \dtgterm{weight map}, where $\wt (x)$ is said to be the \dtgterm{weight} of $x \in Q$, and the maps $\qKoe_i$ and $\qKof_i$ ($i \in I$) are called the \dtgterm{raising} and \dtgterm{lowering quasi-Kashiwara operators}, respectively.
\end{dfn}

In this definition, $\undf$ is an auxilary symbol.
For $x \in Q$, by $\qKoe_i (x) = \undf$ (or $\qKof_i (x) = \undf$) we mean that $\qKoe_i$ (resp., $\qKof_i$) is \dtgterm{undefined} on $x$. On the other hand, we say that $\qKoe_i$ (or $\qKof_i$) is \dtgterm{defined} on $x$ whenever $\qKoe_i (x) \in Q$ (resp., $\qKof_i (x) \in Q$).
Thus, alternatively one can consider the quasi-Kashiwara operators $\qKoe_i$ and $\qKof_i$ ($i \in I$) to be partial maps from $Q$ to $Q$.
When this point of view is more suitable to describe quasi-Kashiwara operators, we will make use of it.

From condition\avoidrefbreak \itmref{dfn:qcwt} above, observe that $\qKoec_i (x) = +\infty$ if and only if $\qKofc_i (x) = +\infty$, for any $x \in Q$ and $i \in I$.

As in this paper we will only consider quasi-crystals that are seminormal, i.e., quasi-crystals satisfying condition\avoidrefbreak \itmref{dfn:qcsn} above, we included it directly in the definition.
In the sequel, when we only write ``quasi-crystal'', we always mean ``seminormal quasi-crystal''.

A \dtgterm{seminormal crystal} is a quasi-crystal $\qcrstQ$ such that $\qKoec_i (x) \neq +\infty \neq \qKofc_i (x)$, for all $x \in Q$ and $i \in I$.
For a crystal, we usually denote $\qKoe_i$, $\qKof_i$, $\qKoec_i$, and $\qKofc_i$ ($i \in I$) by $\Koe_i$, $\Kof_i$, $\Koec_i$, and $\Kofc_i$, respectively.
Also, $\Koe_i$ and $\Kof_i$ ($i \in I$) are called Kashiwara operators.

As we now describe, a seminormal quasi-crystal can be completely encoded in a graph.
See\avoidcitebreak \cite[\S{}~4]{CGM23quasi-crystals} for further background.

\begin{dfn}
\label{dfn:qcg}
Let $\Phi$ be a root system with weight lattice $\Lambda$ and index set $I$ for the simple roots $(\alpha_i)_{i \in I}$.
The \dtgterm{quasi-crystal graph} $\Gamma_\qcrstQ$ of a quasi-crystal $\qcrstQ$ of type $\Phi$ is a $\Lambda$-weighted $I$-labelled directed graph with vertex set $Q$ and an edge $x \lbedge{i} y$ from $x \in Q$ to $y \in Q$ labelled by $i \in I$ whenever $\qKof_i (x) = y$, and a loop on $x \in Q$ labelled by $i \in I$ whenever $\qKoec_i (x) = +\infty$.
For $x \in Q$, let $\Gamma_\qcrstQ (x)$ denote the connected component of $\Gamma_\qcrstQ$ containing the vertex $x$.
\end{dfn}

A \dtgterm{crystal graph} is a quasi-crystal graph that does not have loops.
Thus, we have that $\Gamma_{\qcrstQ}$ is a crystal graph if and only if $\qcrstQ$ is a crystal.

\begin{exa}
\label{exa:qctAn}
Consider the root system of type $\tAn$.
The \dtgterm{standard crystal $\crtAn$ of type $\tAn$} is given as follows.
The underlying set is the ordered set $A_n = \set{1 < 2 < \cdots < n}$.
For $x \in A_n$, the weight of $x$ is $\wt(x) = \vc{e_x}$.
For $i=1,2,\ldots,n-1$, the Kashiwara operators $\Koe_i$ and $\Kof_i$ are only defined on $i+1$ and $i$, respectively, where $\Koe_i (i+1) = i$ and $\Kof_i (i) = i+1$.
Finally, $\Koec_i (x) = \delta_{x, i+1}$ and $\Kofc_i (x) = \delta_{x, i}$,
where $\delta_{k, l} = 1$ if $k = l$, and $\delta_{k, l} = 0$ whenever $k \neq l$.

The crystal graph $\Gamma_{\crtAn}$ is
\[ 1 \lbedge{1} 2 \lbedge{2} 3 \lbedge{3} \cdots \lbedge{n-1} n, \]
where the weight map is defined as above.
\end{exa}

\begin{exa}
\label{exa:qctCn}
Consider the root system of type $\tCn$.
The \dtgterm{standard crystal $\crtCn$ of type $\tCn$} is defined as follows.
The underlying set is the ordered set $C_n = \set{1 < 2 < \cdots < n < \wbar{n} < \wbar{n-1} < \cdots < \wbar{1}}$.
For $x \in \set{1,2,\ldots,n}$, the weight of $x$ is $\wt(x) = \vc{e_x}$,
and the weight of $\wbar{x}$ is $\wt(\wbar{x}) = -\vc{e_x}$.
For $i=1,2,\ldots,n-1$, the Kashiwara operators $\Koe_i$ and $\Kof_i$ are only defined on the following cases:
$\Koe_i (i+1) = i$, $\Koe_i (\wbar{i}) = \wbar{i+1}$, $\Kof_i (i) = i+1$, and $\Kof_i (\wbar{i+1}) = \wbar{i}$.
The Kashiwara operators $\Koe_n$ and $\Kof_n$ are only defined on $\wbar{n}$ and $n$, respectively, where $\Koe_n (\wbar{n}) = n$ and $\Kof_n (n) = \wbar{n}$.
Finally, for $y \in C_n$, $\Koec_i (y) = \delta_{y, i+1} + \delta_{y, \wbar{i}}$, $\Koec_n (y) = \delta_{y, \wbar{n}}$, $\Kofc_i (y) = \delta_{y, i} + \delta_{y, \wbar{i+1}}$, and $\Kofc_n (y) = \delta_{y, n}$.

The crystal graph $\Gamma_{\crtCn}$ is
\[ 1 \lbedge{1} 2 \lbedge{2} \cdots \lbedge{n-1} n \lbedge{n} \wbar{n} \lbedge{n-1} \wbar{n-1} \lbedge{n-2} \cdots \lbedge{1} \wbar{1}, \]
where the weight map is defined as above.
\end{exa}

To conclude, we introduce the notion of homomorphism between seminormal quasi-crystals.

\begin{dfn}
\label{dfn:qch}
Let $\qcrstQ$ and $\qcrstQ'$ be seminormal quasi-crystals of the same type.
A \dtgterm{quasi-crystal homomorphism} $\psi$ from $\qcrstQ$ to $\qcrstQ'$, denoted by $\psi : \qcrstQ \to \qcrstQ'$, is a map $\psi : Q \sqcup \set{\undf} \to Q' \sqcup \set{\undf}$ that satisfies the following conditions:
\begin{enumerate}
\item\label{dfn:qchundf}
$\psi (\undf) = \undf$;

\item\label{dfn:qchwtc}
if $\psi(x) \in Q'$, then
$\wt \parens[\big]{ \psi(x) } = \wt(x)$,
$\qKoec_i \parens[\big]{ \psi(x) } = \qKoec_i (x)$, 
and\linebreak $\qKofc_i \parens[\big]{ \psi(x) } = \qKofc_i (x)$;

\item\label{dfn:qchqKoe}
if $\qKoe_i (x) \in Q$ and $\psi(x), \psi \parens[\big]{ \qKoe_i (x) } \in Q'$, then
$\psi \parens[\big]{ \qKoe_i (x) } = \qKoe_i \parens[\big]{ \psi (x) }$;

\item\label{dfn:qchqKof}
if $\qKof_i (x) \in Q$ and $\psi(x), \psi \parens[\big]{ \qKof_i (x) } \in Q'$, then
$\psi \parens[\big]{ \qKof_i (x) } = \qKof_i \parens[\big]{ \psi (x) }$;
\end{enumerate}
for $x \in Q$ and $i \in I$.

If $\psi$ is also bijective, then it is called a \dtgterm{quasi-crystal isomorphism}.
We say that $\qcrstQ$ and $\qcrstQ'$ are \dtgterm{isomorphic} if there exists a quasi-crystal isomorphism between $\qcrstQ$ and $\qcrstQ'$.
\end{dfn}

In terms of quasi-crystal graphs, we have that two quasi-crystals $\qcrstQ$ and $\qcrstQ'$ of the same type are isomorphic if and only if there exists a graph isomorphism between $\Gamma_{\qcrstQ}$ and $\Gamma_{\qcrstQ'}$ that preserves vertex weights and edge labels.

\section{Plactic monoids over quasi-crystals}
\label{sec:plmqc}

In this section, we present a construction of the plactic monoid associated to a seminormal quasi-crystal.
It is based on the well-known construction of the plactic monoid associated to a seminormal crystal, which has its roots in the work by Kashiwara and Nakashima\avoidcitebreak \cite{KN94} and is commonly found in the literature.
See for example\avoidcitebreak \cite{Lec07,CGM19}.

We first define the tensor product of seminormal quasi-crystals.
We then introduce quasi-crystal monoids for this tensor product and define a free quasi-crystal monoid associated to a seminormal quasi-crystal.
Finally, we show how the plactic monoid can be defined as a quotient of the free quasi-crystal monoid.

\subsection{Tensor product of quasi-crystals}
\label{subsec:qctp}

We present the notion of tensor product of seminormal quasi-crystals following the original notion of tensor product of crystals introduced by Kashiwara\avoidcitebreak \cite{Kas90,Kas91,Kas94}.

\begin{dfn}
\label{dfn:qctp}
Consider a root system $\Phi$ with weight lattice $\Lambda$ and index set $I$ for the simple roots $(\alpha_i)_{i \in I}$.
Let $\qcrstQ$ and $\qcrstQ'$ be seminormal quasi-crystals of type $\Phi$.
The \dtgterm{tensor product} of $\qcrstQ$ and $\qcrstQ'$ is a seminormal quasi-crystal $\qcrstQ \otimes \qcrstQ'$ of type $\Phi$ given as follows.
The underlying set is $Q \otimes Q'$ which denotes the Cartesian product $Q \times Q'$ whose ordered pairs are denoted by $x \otimes x'$ with $x \in Q$ and $x' \in Q'$.
The quasi-crystal structure is defined by
\begin{align*}
\wt (x \otimes x') &= \wt(x) + \wt(x'),
\displaybreak[0]\\
\qKoe_i (x \otimes x') &=
   \begin{cases}
      \qKoe_i (x) \otimes x' & \text{if $\qKofc_i (x) \geq \qKoec_i (x')$}\\
      x \otimes \qKoe_i (x') & \text{if $\qKofc_i (x) < \qKoec_i (x')$,}
   \end{cases}
\displaybreak[0]\\
\qKof_i (x \otimes x') &=
   \begin{cases}
      \qKof_i (x) \otimes x' & \text{if $\qKofc_i (x) > \qKoec_i (x')$}\\
      x \otimes \qKof_i (x') & \text{if $\qKofc_i (x) \leq \qKoec_i (x')$,}
   \end{cases}
\displaybreak[0]\\
\qKoec_i (x \otimes x') &= \max \set[\big]{ \qKoec_i (x), \qKoec_i (x') - \innerp[\big]{\wt (x)}{\alphav_i} },
\displaybreak[0]\\
\shortintertext{and}
\qKofc_i (x \otimes x') &= \max \set[\big]{ \qKofc_i (x) + \innerp[\big]{\wt(x')}{\alphav_i}, \qKofc_i (x') },
\end{align*}
where $x \dotimes \undf = \undf \dotimes x' = \undf$,
for $x \in Q$, $x' \in Q'$, and $i \in I$.
\end{dfn}

From the previous definition, it is immediate that $\qKoec_i (x \otimes x') = +\infty$ (or equivalently, $\qKofc_i (x \otimes x') = +\infty$) if and only if $\qKoec_i (x) = +\infty$ or $\qKoec_i (x') = +\infty$ (or equivalently, $\qKofc_i (x) = +\infty$ or $\qKofc_i (x') = +\infty$).
Thus, the tensor product of seminormal crystals is still a seminormal crystal.
Due to this observation, it is straightforward to adapt the proofs in\avoidcitebreak \cite[\S{}~4.4]{HK02} to check that the tensor product of seminormal quasi-crystals is a seminormal quasi-crystal and that the tensor product of seminormal quasi-crystals is associative, i.e., given seminormal quasi-crystals $\qcrstQ_1$, $\qcrstQ_2$, and $\qcrstQ_3$ of the same type, the quasi-crystals $(\qcrstQ_1 \otimes \qcrstQ_2) \otimes \qcrstQ_3$ and $\qcrstQ_1 \otimes (\qcrstQ_2 \otimes \qcrstQ_3)$ are isomorphic, whereas a quasi-crystal isomorphism is given by $(x_1 \otimes x_2) \otimes x_3 \mapsto x_1 \otimes (x_2 \otimes x_3)$.

\begin{exa}
\label{exa:qctpA32}
The crystal graph $\Gamma_{\crtA_3 \otimes \crtA_3}$ of the tensor product $\crtA_3 \otimes \crtA_3$ is the following.
\[
\begin{tikzpicture}[widecrystal,baseline=(33.base)]
  %
  \node (11) at (1, 3) {1 \otimes 1};
  \node (21) at (2, 3) {2 \otimes 1};
  \node (31) at (3, 3) {3 \otimes 1};
  \node (12) at (1, 2) {1 \otimes 2};
  \node (22) at (2, 2) {2 \otimes 2};
  \node (32) at (3, 2) {3 \otimes 2};
  \node (13) at (1, 1) {1 \otimes 3};
  \node (23) at (2, 1) {2 \otimes 3};
  \node (33) at (3, 1) {3 \otimes 3};
  \path (11) edge node {1} (21)
        (21) edge node {1} (22)
        (21) edge node {2} (31)
        (31) edge node {1} (32)
        (12) edge node {2} (13)
        (22) edge node {2} (32)
        (32) edge node {2} (33)
        (13) edge node {1} (23);
\end{tikzpicture}
\]
where $\wt (x \otimes y) = \wt(x) + \wt(y)$, for $x, y \in A_3$.
\end{exa}

\begin{exa}
\label{exa:qctpC22}
The crystal graph $\Gamma_{\crtC_2 \otimes \crtC_2}$ of the tensor product $\crtC_2 \otimes \crtC_2$ is the following.
\[
\begin{tikzpicture}[widecrystal,baseline=(b1b1.base)]
  %
  \node (11)  at (1, 4) {1 \otimes 1};
  \node (21)  at (2, 4) {2 \otimes 1};
  \node (b21) at (3, 4) {\wbar{2} \otimes 1};
  \node (b11) at (4, 4) {\wbar{1} \otimes 1};
  \node (12)  at (1, 3) {1 \otimes 2};
  \node (22)  at (2, 3) {2 \otimes 2};
  \node (b22) at (3, 3) {\wbar{2} \otimes 2};
  \node (b12) at (4, 3) {\wbar{1} \otimes 2};
  \node (1b2)  at (1, 2) {1 \otimes \wbar{2}};
  \node (2b2)  at (2, 2) {2 \otimes \wbar{2}};
  \node (b2b2) at (3, 2) {\wbar{2} \otimes \wbar{2}};
  \node (b1b2) at (4, 2) {\wbar{1} \otimes \wbar{2}};
  \node (1b1)  at (1, 1) {1 \otimes \wbar{1}};
  \node (2b1)  at (2, 1) {2 \otimes \wbar{1}};
  \node (b2b1) at (3, 1) {\wbar{2} \otimes \wbar{1}};
  \node (b1b1) at (4, 1) {\wbar{1} \otimes \wbar{1}};
  \path (11) edge node {1} (21)
        (21) edge node {1} (22)
        (21) edge node {2} (b21)
        (b21) edge node {1} (b11)
        (b11) edge node {1} (b12)
        (12) edge node {2} (1b2)
        (22) edge node {2} (b22)
        (b22) edge node {2} (b2b2)
        (b12) edge node {2} (b1b2)
        (1b2) edge node {1} (2b2)
        (2b2) edge node {1} (2b1)
        (b2b2) edge node {1} (b1b2)
        (b1b2) edge node {1} (b1b1)
        (2b1) edge node {2} (b2b1);
\end{tikzpicture}
\]
where $\wt (x \otimes y) = \wt(x) + \wt(y)$, for $x, y \in C_2$.
\end{exa}

We chose to base our definition of tensor product of seminormal quasi-crystals on the original notion of tensor product of crystals introduced by Kashiwara\avoidcitebreak \cite{Kas90,Kas91,Kas94}.
There is another convention for the tensor product of crystals which is opposite to the original one (see for example\avoidcitebreak \cite[\S{}~2.3]{BS17}).
The impact of this choice in the subsequent sections is that the monoids considered are anti-isomorphic to the ones that would be obtained if we chose to adopt the other convention.

\subsection{Quasi-crystal monoids for the tensor product}
\label{subsec:qcmtp}

We first introduce the fundamental concept relating quasi-crystals and monoids with respect to the tensor product.

\begin{dfn}
\label{dfn:qcmtp}
Let $\Phi$ be a root system with weight lattice $\Lambda$ and index set $I$ for the simple roots $(\alpha_i)_{i \in I}$.
A \dtgterm{$\otimes$-quasi-crystal monoid} $\qcrstM$ of type $\Phi$ consists of a set $M$ together with maps ${\wt} : M \to \Lambda$, $\qKoe_i, \qKof_i : M \to M \sqcup \set{ \undf }$, $\qKoec_i, \qKofc_i : M \to \Z \cup \set{ +\infty }$ ($i \in I$) and a binary operation ${\cdot} : M \times M \to M$ satisfying the following conditions:
\begin{enumerate}
\item\label{dfn:qcmtpqc}
$M$ together with $\wt$, $\qKoe_i$, $\qKof_i$, $\qKoec_i$ and $\qKofc_i$ ($i \in I$) forms a seminormal quasi-crystal of type $\Phi$;

\item\label{dfn:qcmtpmon}
$M$ together with $\cdot$ forms a monoid;

\item\label{dfn:qcmtphom}
the map $M \otimes M \to M$, given by $x \otimes y \mapsto x \cdot y$ for $x, y \in M$, induces a quasi-crystal homomorphism from $\qcrstM \otimes \qcrstM$ to $\qcrstM$.
\end{enumerate}
\end{dfn}

In a $\otimes$-quasi-crystal monoid $\qcrstM$ the interaction between the quasi-crystal structure and the binary operation satisfies rules similar to those satisfied by the quasi-crystal structure of a tensor product.
That is,
\begin{align*}
\wt (x y) &= \wt(x) + \wt(y),
\displaybreak[0]\\
\qKoe_i (x y) &=
   \begin{cases}
      \qKoe_i (x) \cdot y & \text{if $\qKofc_i (x) \geq \qKoec_i (y)$}\\
      x \cdot \qKoe_i (y) & \text{if $\qKofc_i (x) < \qKoec_i (y)$,}
   \end{cases}
\displaybreak[0]\\
\qKof_i (x y) &=
   \begin{cases}
      \qKof_i (x) \cdot y & \text{if $\qKofc_i (x) > \qKoec_i (y)$}\\
      x \cdot \qKof_i (y) & \text{if $\qKofc_i (x) \leq \qKoec_i (y)$,}
   \end{cases}
\displaybreak[0]\\
\qKoec_i (x y) &= \max \set[\big]{ \qKoec_i (x), \qKoec_i (y) - \innerp[\big]{\wt (x)}{\alphav_i} },
\displaybreak[0]\\
\shortintertext{and}
\qKofc_i (x y) &= \max \set[\big]{ \qKofc_i (x) + \innerp[\big]{\wt(y)}{\alphav_i}, \qKofc_i (y) },
\end{align*}
for $x, y \in M$ and $i \in I$.

When $M$ together with $\wt$, $\qKoe_i$, $\qKof_i$, $\qKoec_i$ and $\qKofc_i$ ($i \in I$) forms a seminormal crystal, we say that $\qcrstM$ is a \dtgterm{$\otimes$-crystal monoid}.

Given a seminormal quasi-crystal $\qcrstQ$, as the tensor product $\otimes$ is an associative operation, denote by $\qcrstQ^{\otimes k}$ the tensor product of $k$ copies of $\qcrstQ$ and by $Q^{\otimes k}$ its underlying set.
the set
\[
  Q^{\otimes \omega} = \bigcup_{k \geq 0} Q^{\otimes k}
\]
inherits a $\otimes$-quasi-crystal monoid structure, which gives rise to the following definition.

\begin{dfn}
\label{dfn:fqcmtp}
Let $\qcrstQ$ be a seminormal quasi-crystal.
The \dtgterm{free $\otimes$-quasi-crystal monoid} $\ftpqcm{\qcrstQ}$ over $\qcrstQ$ is a $\otimes$-quasi-crystal monoid of the same type as $\qcrstQ$ consisting of the set $Q^*$ of all words over $Q$, the usual concatenation of words, and quasi-crystal structure maps defined as follows.
For $i \in I$, set
\[
\wt (\ew) = 0,
\quad
\qKoe_i (\ew) = \qKof_i (\ew) = \undf,
\quad \text{and} \quad
\qKoec_i (\ew) = \qKofc_i (\ew) = 0,
\]
and for $u, v \in Q^*$, set
\begin{align*}
\wt (uv) &= \wt(u) + \wt(v),
\displaybreak[0]\\
\qKoe_i (uv) &=
  \begin{cases}
    \qKoe_i (u) v & \text{if $\qKofc_i (u) \geq \qKoec_i (v)$}\\
    u \qKoe_i (v) & \text{if $\qKofc_i (u) < \qKoec_i (v)$,}
  \end{cases}
\displaybreak[0]\\
\qKof_i (uv) &=
  \begin{cases}
    \qKof_i (u) v & \text{if $\qKofc_i (u) > \qKoec_i (v)$}\\
    u \qKof_i (v) & \text{if $\qKofc_i (u) \leq \qKoec_i (v)$,}
  \end{cases}
\displaybreak[0]\\
\qKoec_i (uv) &= \max \set[\big]{ \qKoec_i (u), \qKoec_i (v) - \innerp[\big]{\wt (u)}{\alphav_i} },
\displaybreak[0]\\
\shortintertext{and}
\qKofc_i (uv) &= \max \set[\big]{ \qKofc_i (u) + \innerp[\big]{\wt(v)}{\alphav_i}, \qKofc_i (v) },
\end{align*}
where $u \undf = \undf v = \undf$.
\end{dfn}

Notice that we explicitly gave the values of the quasi-crystal structure maps of $\ftpqcm{\qcrstQ}$ on the empty word $\ew$, on letters the values follow from the quasi-crystal structure maps of $\qcrstQ$, and on a word of the form $uv$ they depend only on their values on $u$ and $v$.
Thus, the definition of the quasi-crystal structure above is not circular.
Moreover, we can obtain the values of the quasi-crystal structure maps on a nonempty word based only on their values on its letters.
For the weight map, we have that
\[ \wt (x_1 \ldots x_m) = \wt(x_1) + \cdots + \wt(x_m). \]
for any $x_1, \ldots, x_m \in Q$.
For $\qKoe_i$, $\qKof_i$, $\qKoec_i$, and $\qKofc_i$ ($i \in I$), we present a method called \dtgterm{signature rule}.

Let $B_0$ denote the bicyclic monoid with a zero element added.
We have the following presentation
\[
  B_0 = \pres[\big]{ 0, {-}, {+} \given ({+} {-}, \ew) },
\]
where we omitted the relations $(x0, 0)$ and $(0x, 0)$, for $x \in \set{ 0, {-}, {+} }$, for the sake of simplicity.

Let $\qcrstQ$ be a seminormal quasi-crystal.
For each $i \in I$, define a map $\tpsgn_i : Q^* \to B_0$ by
\[
  \tpsgn_i (w) =
  \begin{cases}
    0 & \text{if $\qKoec_i (w) = +\infty$}\\
    {-^{\qKoec_i (w)}} {+^{\qKofc_i (w)}} & \text{otherwise,}
  \end{cases}
\]
for $w \in Q^*$.
The map $\tpsgn_i$ is called the \dtgterm{$i$-signature map} for the tensor product $\otimes$.

The $i$-signature map $\tpsgn_i$ is a monoid homomorphism.
Thus, given a word $w = x_1 \ldots x_m$ with $x_1, \ldots, x_m \in Q$, we have that
\[ \tpsgn_i (w) = \tpsgn_i (x_1 \ldots x_m) = \tpsgn_i (x_1) \cdots \tpsgn_i (x_m). \]
If $\tpsgn_i (w) = 0$, then $\qKoec_i (w) = \qKofc_i (w) = +\infty$ which implies $\qKoe_i (w) = \qKof_i (w) = \undf$.
Otherwise,
$\tpsgn_i (w) = {-^{a}} {+^{b}}$,
for some $a, b \in \Z_{\geq 0}$.
Then, $\qKoec_i (w) = a$ and $\qKofc_i (w) = b$.
If $a \geq 1$, then
$\qKoe_i (w) = x_1 \ldots x_{p-1} \qKoe_i (x_p) x_{p+1} \ldots x_m$,
where $x_p$ originates the right-most symbol $-$ in $\tpsgn_i (w) = {-^{a-1}} {-} {+^b}$.
If $b \geq 1$, then
$\qKof_i (w) = x_1 \ldots x_{q-1} \qKof_i (x_q) x_{q+1} \ldots x_m$,
where $x_q$ originates the left-most symbol $+$ in $\tpsgn_i (w) = {-^a} {+} {+^{b-1}}$.

\begin{exa}
\label{exa:tpsgnr}
Consider the standard crystal $\crtA_4$ of type $\tA_4$.
For $i \in \set{1,2,3}$, we have that $\tpsgn_i (i) = {+}$, $\tpsgn_i (i+1) = {-}$, and $\tpsgn_i (x) = \ew$, for any $x \in A_4 \setminus \set{i,i+1}$.

We compute $\Koe_i$, $\Kof_i$, $\Koec_i$, and $\Kofc_i$ on $w = 3 1 3 1 2 2 4 1 4$, for $i \in \set{1,2,3}$, using the signature rule.
To keep track to which element originates each $-$ and $+$ we write a subscript with the position of the letter, this is just an auxiliary notation and the binary operation of $B_0$ should be applied ignoring the subscripts.

For $i=1$, we have that
\[ \tpsgn_1 (w) = \tpsgn_1 (3 1 3 1 2 2 4 1 4) = {+_2} {+_4} {-_5} {-_6} {+_8} = {+_8}. \]
Thus, $\Koec_1 (w) = 0$ which implies that $\Koe_1$ is undefined on $w$, and $\Kofc_1 (w) = 1$ where $\Kof_1 (w)$ is obtain  by applying $\Kof_1$ to the letter in the 8th position, resulting in
$\Kof_1 (w) = 3 1 3 1 2 2 4 2 4$.

For $i=2$, we have that
\[ \tpsgn_2 (w) = \tpsgn_2 (3 1 3 1 2 2 4 1 4) = {-_1} {-_3} {+_5} {+_6}. \]
Therefore, $\Koec_2 (w) = 2$, $\Koe_2 (w)$ is obtained by applying $\Koe_2$ to the 3rd letter resulting in $\Koe_2 (w) = 3 1 2 1 2 2 4 1 4$, $\Kof_2 (w) = 2$, and $\Kof_2 (w) = 3 1 3 1 3 2 4 1 4$ results from applying $\Kof_2$ to the 5th letter.

For $i=3$, we have that
\[ \tpsgn_3 (w) = \tpsgn_3 (3 1 3 1 2 2 4 1 4) = {+_1} {+_3} {-_7} {-_9} = \ew, \]
which implies that $\Koec_3 (w) = \Kofc_3 (w) = 0$ and $\Koe_3 (w) = \Kof_3 (w) = \undf$, that is, $\Koe_3$ and $\Kof_3$ are undefined on $w$.
\end{exa}

\subsection{Plactic monoids}
\label{subsec:plm}

Given a seminormal quasi-crystal $\qcrstQ$ and an element $x \in Q$, we denote by $\qcrstQ (x)$ the seminormal quasi-crystal whose quasi-crystal graph corresponds to the connected component $\Gamma_{\qcrstQ} (x)$ of $\Gamma_{\qcrstQ}$ containing $x$.
We call $\qcrstQ (x)$ the \dtgterm{connected component} of $\qcrstQ$ containing $x$, and denote its underlying set by $Q(x)$.
Thus, the quasi-crystal structure of $\qcrstQ (x)$ corresponds to the restriction of the quasi-crystal structure of $\qcrstQ$ to $Q(x)$.

\begin{dfn}
\label{dfn:plco}
Let $\qcrstQ$ be a seminormal quasi-crystal.
The \dtgterm{plactic congruence} on $\ftpqcm{\qcrstQ}$ is a relation $\plco$ on $Q^*$ given as follows.
For $u, v \in Q^*$, $u \plco v$ if and only if there exists a quasi-crystal isomorphism $\psi : \ftpqcm{\qcrstQ} (u) \to \ftpqcm{\qcrstQ} (v)$ such that $\psi (u) = v$.
\end{dfn}

Due to the following result, we can define the plactic monoid over a quasi-crystal.

\begin{prop}
\label{prop:plcomoncong}
Let $\qcrstQ$ be a seminormal quasi-crystal.
Then, the plactic congruence $\plco$ on $\ftpqcm{\qcrstQ}$ is a monoid congruence on the free monoid $Q^*$.
\end{prop}

\begin{dfn}
\label{dfn:placmon}
Let $\qcrstQ$ be a seminormal quasi-crystal, and let $\plco$ be the plactic congruence on $\ftpqcm{\qcrstQ}$.
The quotient monoid $Q^* / {\plco}$ is called the \dtgterm{plactic monoid} associated to $\qcrstQ$ and is denoted by $\plac (\qcrstQ)$.
\end{dfn}

Observe that the quasi-crystal structure of $\ftpqcm{\qcrstQ}$ induces a quasi-crystal structure on $\plac (\qcrstQ)$, since if $u \plco v$, then $\wt(u) = \wt(v)$, $\qKoec_i (u) = \qKoec_i (v)$, $\qKofc_i (u) = \qKofc_i (v)$, $\qKoe_i (u) \plco \qKoe_i (v)$ whenever $\qKoe_i$ is defined on $u$, and $\qKof_i (u) \plco \qKof_i (v)$ whenever $\qKof_i$ is defined on $v$ ($i \in I$).
Although $\plac(\qcrstQ)$ is a $\otimes$-quasi-crystal monoid, we simply call it the plactic monoid, because we are mainly interested in studying it as a monoid.
However, we will be constantly considering its quasi-crystal structure, as it plays a fundamental role in the construction of $\plac(\qcrstQ)$, and consequently, in its properties.

For type $\tAn$, we have that the plactic monoid $\plac (\crtAn)$ is anti-isomorphic to the classical plactic monoid introduced by Lascoux and Sch\"{u}tzenberger\avoidcitebreak \cite{LS81}.

\begin{exa}
\label{exa:plA3}
Consider the standard crystal $\crtA_3$ of type $\tA_3$.
The connected component $\Gamma_{\ftpqcm{\crtA_3}} (112)$
\[
\begin{tikzpicture}[widecrystal,baseline=(223.base)]
  %
  \node (112) at (1, 1.5) {112};
  \node (212) at (2, 2) {212};
  \node (312) at (3, 2) {312};
  \node (313) at (4, 2) {313};
  \node (113) at (2, 1) {113};
  \node (213) at (3, 1) {213};
  \node (223) at (4, 1) {223};
  \node (323) at (5, 1.5) {323};
  \path (112) edge node {1} (212)
        (212) edge node {2} (312)
        (312) edge node {2} (313)
        (313) edge node {1} (323)
        (112) edge[swap] node {2} (113)
        (113) edge[swap] node {1} (213)
        (213) edge[swap] node {1} (223)
        (223) edge[swap] node {2} (323);
\end{tikzpicture}
\]
and the connected component $\Gamma_{\ftpqcm{\crtA_3}} (121)$
\[
\begin{tikzpicture}[widecrystal,baseline=(223.base)]
  %
  \node (112) at (1, 1.5) {121};
  \node (212) at (2, 2) {122};
  \node (312) at (3, 2) {132};
  \node (313) at (4, 2) {133};
  \node (113) at (2, 1) {131};
  \node (213) at (3, 1) {231};
  \node (223) at (4, 1) {232};
  \node (323) at (5, 1.5) {233};
  \path (112) edge node {1} (212)
        (212) edge node {2} (312)
        (312) edge node {2} (313)
        (313) edge node {1} (323)
        (112) edge[swap] node {2} (113)
        (113) edge[swap] node {1} (213)
        (213) edge[swap] node {1} (223)
        (223) edge[swap] node {2} (323);
\end{tikzpicture}
\]
are isomorphic.
For instance, we have that $112 \plco 121$, $212 \plco 122$, $223 \plco 232$, and $323 \plco 233$.
\end{exa}

\begin{exa}
\label{exa:plC3}
Consider the standard crystal $\crtC_3$ of type $\tC_3$.
The connected component $\Gamma_{\ftpqcm{\crtC_3}} \parens[\big]{ 12\wbar{2} }$
\[
  12\wbar{2} \lbedge{1} 12\wbar{1} \lbedge{2} 13\wbar{1} \lbedge{3} 1\wbar{3}\,\wbar{1} \lbedge{2} 1\wbar{2}\,\wbar{1} \lbedge{1} 2\wbar{2}\,\wbar{1}
\]
is isomorphic to the connected component $\Gamma_{\ftpqcm{\crtC_3}} (1)$, described in \comboref{Example}{exa:qctCn}.
For instance, we get that 
$12\wbar{2} \plco 1$,
$13\wbar{1} \plco 3$, 
and $2\wbar{2}\,\wbar{1} \plco \wbar{1}$.
\end{exa}

\section{Hypoplactic monoids over quasi-crystals}
\label{sec:hypomqc}

In this section, we present a construction of the hypoplactic monoid associated to a seminormal quasi-crystal.
It is analogous to the construction described in \comboref{Section}{sec:plmqc}.
For a detailed study of this construction and for the proofs of all results in this section, see\avoidcitebreak \cite{CGM23quasi-crystals}.

We first define the quasi-tensor product of seminormal quasi-crystals.
We then introduce quasi-crystal monoids and free quasi-crystal monoids for the quasi-tensor product.
Finally, we show how the plactic monoid can be defined as a quotient of the free quasi-crystal monoid.

\subsection{Quasi-tensor product of quasi-crystals}
\label{subsec:qcqtp}

We present the notion of quasi-tensor product of seminormal quasi-crystals.
A detailed study of this notion can be found in\avoidcitebreak \cite[\S{}~5]{CGM23quasi-crystals}.

\begin{dfn}
\label{dfn:qcqtp}
Consider a root system $\Phi$ with weight lattice $\Lambda$ and index set $I$ for the simple roots $(\alpha_i)_{i \in I}$.
Let $\qcrstQ$ and $\qcrstQ'$ be seminormal quasi-crystals of type $\Phi$.
The \dtgterm{inverse-free quasi-tensor product of $\qcrstQ$ and $\qcrstQ'$}, or simply the \dtgterm{quasi-tensor product of $\qcrstQ$ and $\qcrstQ'$}, is the seminormal quasi-crystal $\qcrstQ \dotimes \qcrstQ'$ given as follows.
The underlying set is $Q \dotimes Q'$ which consists of the Cartesian product $Q \times Q'$ whose ordered pairs are denoted by $x \dotimes x'$ with $x \in Q$ and $x' \in Q'$.
The quasi-crystal structure is defined by
\[
\wt (x \dotimes x') = \wt(x) + \wt(x'),
\]
and
\begin{enumerate}
\item\label{thm:qcqtpif}
if $\qKofc_i (x) > 0$ and $\qKoec_i (x') > 0$, then
\[
\qKoe_i (x \dotimes x') = \qKof_i (x \dotimes x') = \undf
\quad \text{and} \quad
\qKoec_i (x \dotimes x') = \qKofc_i (x \dotimes x') = {+\infty};
\]

\item\label{thm:qcqtpot}
otherwise,
\begin{align*}
\qKoe_i (x \dotimes x') &=
   \begin{cases}
      \qKoe_i (x) \dotimes x' & \text{if $\qKofc_i (x) \geq \qKoec_i (x')$}\\
      x \dotimes \qKoe_i (x') & \text{if $\qKofc_i (x) < \qKoec_i (x')$,}
   \end{cases}
\displaybreak[0]\\
\qKof_i (x \dotimes x') &=
   \begin{cases}
      \qKof_i (x) \dotimes x' & \text{if $\qKofc_i (x) > \qKoec_i (x')$}\\
      x \dotimes \qKof_i (x') & \text{if $\qKofc_i (x) \leq \qKoec_i (x')$,}
   \end{cases}
\displaybreak[0]\\
\qKoec_i (x \dotimes x') &= \max \set[\big]{ \qKoec_i (x), \qKoec_i (x') - \innerp[\big]{\wt (x)}{\alphav_i} },
\displaybreak[0]\\
\shortintertext{and}
\qKofc_i (x \dotimes x') &= \max \set[\big]{ \qKofc_i (x) + \innerp[\big]{\wt(x')}{\alphav_i}, \qKofc_i (x') },
\end{align*}
where $x \dotimes \undf = \undf \dotimes x' = \undf$;
\end{enumerate}
for $x \in Q$, $x' \in Q'$, and $i \in I$.
\end{dfn}

The quasi-tensor product of seminormal quasi-crystals is associative, that is, given seminormal quasi-crystals $\qcrstQ_1$, $\qcrstQ_2$, and $\qcrstQ_3$ of the same type, the quasi-crystals $(\qcrstQ_1 \dotimes \qcrstQ_2) \dotimes \qcrstQ_3$ and $\qcrstQ_1 \dotimes (\qcrstQ_2 \dotimes \qcrstQ_3)$ are isomorphic, whereas a quasi-crystal isomorphism is given by $(x_1 \dotimes x_2) \dotimes x_3 \mapsto x_1 \dotimes (x_2 \dotimes x_3)$.

Observe that the quasi-tensor product of seminormal crystals may not be a seminormal crystal.
In fact, given a crystal $\qcrstQ$, if $x \in Q$ and $i \in I$ are such that $\Kof_i$ is defined on $x$, then $\Kofc_i (x) > 0$ and $\Koec_i \parens[\big]{ \Kof_i (x) } = \Koec_i (x) + 1 > 0$, which implies that $\qKoec_i \parens[\big]{ x \dotimes \Kof_i (x) } = +\infty$, and so $\qcrstQ \dotimes \qcrstQ$ is not a crystal.

Comparing the notions of tensor product (\comboref{Definition}{dfn:qctp}) and quasi-tensor product, we have the following observation.

\begin{rmk}
\label{rmk:qctpqtp}
Let $\qcrstQ$ and $\qcrstQ'$ be seminormal quasi-crystals of the same type, and let $x \in Q$ and $x' \in Q'$.
The weight maps of $\qcrstQ \otimes \qcrstQ'$ and $\qcrstQ \dotimes \qcrstQ'$ coincide, i.e., $\wt (x \otimes x') = \wt (x \dotimes x')$.
If $\qKoec_i (x \dotimes x') \neq +\infty$ (or equivalently, $\qKofc_i (x \dotimes x') \neq +\infty$), then $\qKoec_i (x \dotimes x') = \qKoec_i (x \otimes x')$ and $\qKofc_i (x \dotimes x') = \qKofc_i (x \otimes x')$.
If $\qKoe_i$ (or $\qKof_i$) is defined on $x \dotimes x'$, then $\qKoe_i$ (resp., $\qKof_i$) is defined on $x \otimes x'$; moreover, if $\qKoe_i (x \dotimes x') = y \dotimes y'$ (or $\qKof_i (x \dotimes x') = z \dotimes z'$), then $\qKoe_i (x \otimes x') = y \otimes y'$ (resp., $\qKof_i (x \otimes x') = z \otimes z'$).
\end{rmk}

\begin{exa}
\label{exa:qcqtpA32}
The quasi-crystal graph $\Gamma_{\crtA_3 \dotimes \crtA_3}$ of the quasi-tensor product $\crtA_3 \dotimes \crtA_3$ is the following.
\[
\begin{tikzpicture}[widecrystal,baseline=(33.base)]
  %
  \node (11) at (1, 3) {1 \dotimes 1};
  \node (21) at (2, 3) {2 \dotimes 1};
  \node (31) at (3, 3) {3 \dotimes 1};
  \node (12) at (1, 2) {1 \dotimes 2};
  \node (22) at (2, 2) {2 \dotimes 2};
  \node (32) at (3, 2) {3 \dotimes 2};
  \node (13) at (1, 1) {1 \dotimes 3};
  \node (23) at (2, 1) {2 \dotimes 3};
  \node (33) at (3, 1) {3 \dotimes 3};
  \path (11) edge node {1} (21)
        (21) edge node {1} (22)
        (21) edge node {2} (31)
        (31) edge node {1} (32)
        (12) edge [loop left] node {1} ()
        (12) edge node {2} (13)
        (22) edge node {2} (32)
        (32) edge node {2} (33)
        (13) edge node {1} (23)
        (23) edge [loop right] node {2} ();
\end{tikzpicture}
\]
where $\wt (x \dotimes y) = \wt(x) + \wt(y)$, for $x, y \in A_3$.
\end{exa}

\begin{exa}
\label{exa:qcqtpC22}
The quasi-crystal graph $\Gamma_{\qctC_2 \dotimes \qctC_2}$ of the quasi-tensor product $\qctC_2 \dotimes \qctC_2$ is the following.
\[
\begin{tikzpicture}[widecrystal,baseline=(b1b1.base)]
  %
  \node (11)  at (1, 4) {1 \dotimes 1};
  \node (21)  at (2, 4) {2 \dotimes 1};
  \node (b21) at (3, 4) {\wbar{2} \dotimes 1};
  \node (b11) at (4, 4) {\wbar{1} \dotimes 1};
  \node (12)  at (1, 3) {1 \dotimes 2};
  \node (22)  at (2, 3) {2 \dotimes 2};
  \node (b22) at (3, 3) {\wbar{2} \dotimes 2};
  \node (b12) at (4, 3) {\wbar{1} \dotimes 2};
  \node (1b2)  at (1, 2) {1 \dotimes \wbar{2}};
  \node (2b2)  at (2, 2) {2 \dotimes \wbar{2}};
  \node (b2b2) at (3, 2) {\wbar{2} \dotimes \wbar{2}};
  \node (b1b2) at (4, 2) {\wbar{1} \dotimes \wbar{2}};
  \node (1b1)  at (1, 1) {1 \dotimes \wbar{1}};
  \node (2b1)  at (2, 1) {2 \dotimes \wbar{1}};
  \node (b2b1) at (3, 1) {\wbar{2} \dotimes \wbar{1}};
  \node (b1b1) at (4, 1) {\wbar{1} \dotimes \wbar{1}};
  \path (11) edge node {1} (21)
        (21) edge node {1} (22)
        (21) edge node {2} (b21)
        (b21) edge node {1} (b11)
        (b11) edge node {1} (b12)
        (12) edge node {2} (1b2)
        (22) edge node {2} (b22)
        (b22) edge node {2} (b2b2)
        (b12) edge node {2} (b1b2)
        (1b2) edge node {1} (2b2)
        (2b2) edge node {1} (2b1)
        (b2b2) edge node {1} (b1b2)
        (b1b2) edge node {1} (b1b1)
        (2b1) edge node {2} (b2b1)
        (12) edge [loop left] node {1} ()
        (b22) edge [loop right] node {1} ()
        (2b2) edge [loop right] node {2} ()
        (1b1) edge [loop left] node {1} ()
        (b2b1) edge [loop right] node {1} ();
\end{tikzpicture}
\]
where $\wt (x \dotimes y) = \wt(x) + \wt(y)$, for $x, y \in C_2$.
\end{exa}

\subsection{Quasi-crystal monoids for the quasi-tensor product}
\label{subsec:qcmqtp}

We first introduce the fundamental concept relating quasi-crystals and monoids with respect to the quasi-tensor product.

\begin{dfn}
\label{dfn:qcmqtp}
Let $\Phi$ be a root system with weight lattice $\Lambda$ and index set $I$ for the simple roots $(\alpha_i)_{i \in I}$.
A \dtgterm{$\dotimes$-quasi-crystal monoid} $\qcrstM$ of type $\Phi$ consists of a set $M$ together with maps ${\wt} : M \to \Lambda$, $\qKoe_i, \qKof_i : M \to M \sqcup \set{ \undf }$, $\qKoec_i, \qKofc_i : M \to \Z \cup \set{ +\infty }$ ($i \in I$) and a binary operation ${\cdot} : M \times M \to M$ satisfying the following conditions:
\begin{enumerate}
\item\label{dfn:qcmqtpqc}
$M$ together with $\wt$, $\qKoe_i$, $\qKof_i$, $\qKoec_i$ and $\qKofc_i$ ($i \in I$) forms a seminormal quasi-crystal of type $\Phi$;

\item\label{dfn:qcmqtpmon}
$M$ together with $\cdot$ forms a monoid;

\item\label{dfn:qcmqtphom}
the map $M \dotimes M \to M$, given by $x \dotimes y \mapsto x \cdot y$ for $x, y \in M$, induces a quasi-crystal homomorphism from $\qcrstM \dotimes \qcrstM$ to $\qcrstM$.
\end{enumerate}
\end{dfn}

In a $\dotimes$-quasi-crystal monoid $\qcrstM$ the interaction between the quasi-crystal structure and the binary operation satisfies rules similar to those satisfied by the quasi-crystal structure of a tensor product.
That is, for $x, y \in M$ and $i \in I$,
\[ \wt (x y) = \wt(x) + \wt(y); \]
if $\qKofc_i (x) > 0$ and $\qKoec_i (x') > 0$, then
\[
\qKoe_i (x y) = \qKof_i (x y) = \undf
\quad \text{and} \quad
\qKoec_i (x y) = \qKofc_i (x y) = {+\infty};
\]
otherwise,
\begin{align*}
\qKoe_i (x y) &=
   \begin{cases}
      \qKoe_i (x) \cdot y & \text{if $\qKofc_i (x) \geq \qKoec_i (y)$}\\
      x \cdot \qKoe_i (y) & \text{if $\qKofc_i (x) < \qKoec_i (y)$,}
   \end{cases}
\displaybreak[0]\\
\qKof_i (x y) &=
   \begin{cases}
      \qKof_i (x) \cdot y & \text{if $\qKofc_i (x) > \qKoec_i (y)$}\\
      x \cdot \qKof_i (y) & \text{if $\qKofc_i (x) \leq \qKoec_i (y)$,}
   \end{cases}
\displaybreak[0]\\
\qKoec_i (x y) &= \max \set[\big]{ \qKoec_i (x), \qKoec_i (y) - \innerp[\big]{\wt (x)}{\alphav_i} },
\displaybreak[0]\\
\shortintertext{and}
\qKofc_i (x y) &= \max \set[\big]{ \qKofc_i (x) + \innerp[\big]{\wt(y)}{\alphav_i}, \qKofc_i (y) },
\end{align*}

Given a seminormal quasi-crystal $\qcrstQ$, as the quasi-tensor product $\dotimes$ is an associative operation, denote by $\qcrstQ^{\dotimes k}$ the quasi-tensor product of $k$ copies of $\qcrstQ$ and by $Q^{\dotimes k}$ its underlying set.
the set
\[
  Q^{\dotimes \omega} = \bigcup_{k \geq 0} Q^{\dotimes k}
\]
inherits a $\dotimes$-quasi-crystal monoid structure, which gives rise to the following definition.

\begin{dfn}
\label{dfn:fqcmqtp}
Let $\qcrstQ$ be a seminormal quasi-crystal.
The \dtgterm{free $\dotimes$-quasi-crystal monoid} $\ftpqcm{\qcrstQ}$ over $\qcrstQ$ is a $\dotimes$-quasi-crystal monoid of the same type as $\qcrstQ$ consisting of the set $Q^*$ of all words over $Q$, the usual concatenation of words, and quasi-crystal structure maps defined as follows.
For $i \in I$, set
\[
\wt (\ew) = 0,
\quad
\qKoe_i (\ew) = \qKof_i (\ew) = \undf,
\quad \text{and} \quad
\qKoec_i (\ew) = \qKofc_i (\ew) = 0.
\]
For $u, v \in Q^*$, set
\[ \wt (uv) = \wt(u) + \wt(v); \]
if $\qKofc_i (u) > 0$ and $\qKoec_i (v) > 0$, set
\[
\qKoe_i (u v) = \qKof_i (u v) = \undf
\quad \text{and} \quad
\qKoec_i (u v) = \qKofc_i (u v) = +\infty;
\]
otherwise, set
\begin{align*}
\qKoe_i (uv) &=
  \begin{cases}
    \qKoe_i (u) v & \text{if $\qKofc_i (u) \geq \qKoec_i (v)$}\\
    u \qKoe_i (v) & \text{if $\qKofc_i (u) < \qKoec_i (v)$,}
  \end{cases}
\displaybreak[0]\\
\qKof_i (uv) &=
  \begin{cases}
    \qKof_i (u) v & \text{if $\qKofc_i (u) > \qKoec_i (v)$}\\
    u \qKof_i (v) & \text{if $\qKofc_i (u) \leq \qKoec_i (v)$,}
  \end{cases}
\displaybreak[0]\\
\qKoec_i (uv) &= \max \set[\big]{ \qKoec_i (u), \qKoec_i (v) - \innerp[\big]{\wt (u)}{\alphav_i} },
\displaybreak[0]\\
\shortintertext{and}
\qKofc_i (uv) &= \max \set[\big]{ \qKofc_i (u) + \innerp[\big]{\wt(v)}{\alphav_i}, \qKofc_i (v) };
\end{align*}
where $u \undf = \undf v = \undf$.
\end{dfn}

In a free $\dotimes$-quasi-crystal monoid $\fqtpqcm{\qcrstQ}$, we can obtain the values of the quasi-crystal structure maps on a nonempty word based only on their values on its letters, which are given by the quasi-crystal structure of $\qcrstQ$.
For the weight map, we have that
\[ \wt (x_1 \ldots x_m) = \wt(x_1) + \cdots + \wt(x_m). \]
for any $x_1, \ldots, x_m \in Q$.
For $\qKoe_i$, $\qKof_i$, $\qKoec_i$, and $\qKofc_i$ ($i \in I$), we describe a \dtgterm{signature rule} process.

Consider the zero monoid $Z_0$ with the following presentation
\[
  Z_0 = \pres[\big]{ 0, {-}, {+} \given ({+} {-}, 0) },
\]
where we omitted the relations $(x0, 0)$ and $(0x, 0)$, for $x \in \set{ 0, {-}, {+} }$, for the sake of simplicity.

Let $\qcrstQ$ be a seminormal quasi-crystal.
For each $i \in I$, define a map $\qtpsgn_i : Q^* \to Z_0$ by
\[
  \qtpsgn_i (w) =
  \begin{cases}
    0 & \text{if $\qKoec_i (w) = +\infty$}\\
    {-^{\qKoec_i (w)}} {+^{\qKofc_i (w)}} & \text{otherwise,}
  \end{cases}
\]
for $w \in Q^*$.
The map $\qtpsgn_i$ is called the \dtgterm{$i$-signature map} for the quasi-tensor product $\dotimes$.

The $i$-signature map $\qtpsgn_i$ is a monoid homomorphism.
Thus, given a word $w = x_1 \ldots x_m$ with $x_1, \ldots, x_m \in Q$, we have that
\[ \qtpsgn_i (w) = \qtpsgn_i (x_1 \ldots x_m) = \qtpsgn_i (x_1) \cdots \qtpsgn_i (x_m). \]
If $\qtpsgn_i (w) = 0$, then $\qKoec_i (w) = \qKofc_i (w) = +\infty$ which implies $\qKoe_i (w) = \qKof_i (w) = \undf$.
Otherwise,
$\qtpsgn_i (w) = {-^{a}} {+^{b}}$, for some $a, b \in \Z_{\geq 0}$.
Then, $\qKoec_i (w) = a$ and $\qKofc_i (w) = b$.
If $a \geq 1$, then
$\qKoe_i (w) = x_1 \ldots x_{p-1} \qKoe_i (x_p) x_{p+1} \ldots x_m$,
where $x_p$ originates the right-most symbol $-$ in $\qtpsgn_i (w) = {-^{a-1}} {-} {+^b}$.
If $b \geq 1$, then
$\qKof_i (w) = x_1 \ldots x_{q-1} \qKof_i (x_q) x_{q+1} \ldots x_m$,
where $x_q$ originates the left-most symbol $+$ in $\qtpsgn_i (w) = {-^a} {+} {+^{b-1}}$.

\begin{exa}
\label{exa:qtpsgnr}
Consider the standard crystal $\crtC_3$ of type $\tC_3$.
For $i \in \set{1,2,3}$, we have that $\qtpsgn_i (i) = \qtpsgn_i \parens[\big]{ \wbar{i+1} } = {+}$, $\tpsgn_i (i+1) = \qtpsgn_i \parens[\big]{ \wbar{i} } = {-}$, and $\qtpsgn_i (x) = \ew$, for any $x \in C_3 \setminus \set[\big]{ i, i+1, \wbar{i+1}, \wbar{i} }$.

We compute $\qKoe_i$, $\qKof_i$, $\qKoec_i$, and $\qKofc_i$ on $w = 2 \wbar{1} 1 \wbar{2} 1 3$, for $i \in \set{1,2,3}$, using the signature rule for the quasi-tensor product $\dotimes$.
To keep track to which element originates each $-$ and $+$ we write a subscript with the position of the letter, this is just an auxiliary notation and the binary operation of $Z_0$ should be applied ignoring the subscripts.

For $i=1$, we have that
\[ \qtpsgn_1 (w) = \qtpsgn_1 \parens[\big]{ 2 \wbar{1} 1 \wbar{2} 1 3 } = {-_1} {-_2} {+_3} {+_4} {+_5}. \]
Thus, $\qKoe_1 (w) = 2 \wbar{2} 1 \wbar{2} 1 3$, $\qKof_1 (w) = 2 \wbar{1} 2 \wbar{2} 1 3$, $\qKoec_1 (w) = 2$, and $\qKofc_1 (w) = 3$.

For $i=2$, we have that
\[ \qtpsgn_2 (w) = \qtpsgn_2 \parens[\big]{ 2 \wbar{1} 1 \wbar{2} 1 3 } = {+_1} {-_4} {-_6} = 0. \]
Therefore, $\qKoec_2 (w) = \qKofc_2 (w) = +\infty$ which implies that $\qKoe_2$ and $\qKof_2$ are undefined on $w$.

For $i=3$, we have that
\[ \qtpsgn_3 (w) = \qtpsgn_3 \parens[\big]{ 2 \wbar{1} 1 \wbar{2} 1 3 } = {+_6}. \]
Hence, $\qKoec_3 (w) = 0$ which implies that $\qKoe_3$ is undefined on $w$, and $\qKof_3 (w) = 2 \wbar{1} 1 \wbar{2} 1 \wbar{3}$ and $\qKofc_3 (w) = 1$.
\end{exa}

\subsection{Hypoplactic monoids}
\label{subsec:hypom}

We now show how the hypoplactic monoid arises by identifying words in the same position of isomorphic connected components of a free $\dotimes$-quasi-crystal monoid.

\begin{dfn}
\label{dfn:hypoco}
Let $\qcrstQ$ be a seminormal quasi-crystal.
The \dtgterm{hypoplactic congruence} on $\fqtpqcm{\qcrstQ}$ is a relation $\hyco$ on $Q^*$ given as follows.
For $u, v \in Q^*$, $u \hyco v$ if and only if there exists a quasi-crystal isomorphism $\psi : \fqtpqcm{\qcrstQ} (u) \to \fqtpqcm{\qcrstQ} (v)$ such that $\psi (u) = v$.
\end{dfn}

Due to the following result, we can define the hypoplactic monoid over any seminormal quasi-crystal.

\begin{prop}[{\cite[Theorem~8.23]{Gui22}}]
\label{prop:hypocomoncong}
Let $\qcrstQ$ be a seminormal quasi-crystal.
Then, the hypoplactic congruence $\hyco$ on $\fqtpqcm{\qcrstQ}$ is a monoid congruence on the free monoid $Q^*$.
\end{prop}

\begin{dfn}
\label{dfn:hypomon}
Let $\qcrstQ$ be a seminormal quasi-crystal, and let $\hyco$ be the hypoplactic congruence on $\fqtpqcm{\qcrstQ}$.
The quotient monoid $Q^* / {\hyco}$ is called the \dtgterm{hypoplactic monoid} associated to $\qcrstQ$ and is denoted by $\hypo (\qcrstQ)$.
\end{dfn}

As we pointed out for the plactic monoid, we have that the quasi-crystal structure of $\fqtpqcm{\qcrstQ}$ induces a quasi-crystal structure on $\hypo (\qcrstQ)$.
Although $\hypo(\qcrstQ)$ is a $\dotimes$-quasi-crystal monoid, we simply call it the plactic monoid, because we are mainly interested in studying it as a monoid.
However, we will be constantly considering its quasi-crystal structure, as it plays a fundamental role in the construction of $\hypo(\qcrstQ)$, and consequently, in its properties.

For type $\tAn$, we have that the hypoplactic monoid $\hypo (\crtAn)$ is anti-isomorphic to the classical hypoplactic monoid introduced by Krob and Thibon\avoidcitebreak \cite{KT97}.

\begin{exa}
\label{exa:hypoA3}
Consider the standard crystal $\crtA_3$ of type $\tA_3$.
The connected component $\Gamma_{\fqtpqcm{\crtA_3}} (2121)$
\[
\begin{tikzpicture}[widecrystal,baseline=(223.base)]
  %
  \node (2121) at (0, 0) {2121};
  \node (3121) at (1, 0) {3121};
  \node (3131) at (2, 0) {3131};
  \node (3231) at (3, 0) {3231};
  \node (3232) at (4, 0) {3232};
  \path (2121) edge node {2} (3121)
        (3121) edge node {2} (3131)
        (3131) edge node {1} (3231)
        (3231) edge node {1} (3232)
        (2121) edge [loop above] node {1} ()
        (3121) edge [loop above] node {1} ()
        (3231) edge [loop above] node {2} ()
        (3232) edge [loop above] node {2} ();
\end{tikzpicture}
\]
and the connected component $\Gamma_{\fqtpqcm{\crtA_3}} (122)$
\[
\begin{tikzpicture}[widecrystal,baseline=(223.base)]
  %
  \node (2121) at (0, 0) {1212};
  \node (3121) at (1, 0) {1312};
  \node (3131) at (2, 0) {1313};
  \node (3231) at (3, 0) {2313};
  \node (3232) at (4, 0) {2323};
  \path (2121) edge node {2} (3121)
        (3121) edge node {2} (3131)
        (3131) edge node {1} (3231)
        (3231) edge node {1} (3232)
        (2121) edge [loop above] node {1} ()
        (3121) edge [loop above] node {1} ()
        (3231) edge [loop above] node {2} ()
        (3232) edge [loop above] node {2} ();
\end{tikzpicture}
\]
are isomorphic.
We have that
$2121 \hyco 1212$,
$3121 \hyco 1312$,
$3131 \hyco 1313$,
$3231 \hyco 2313$,
and $3232 \hyco 2323$.
\end{exa}

\begin{exa}
\label{exa:hypoCn}
Consider the standard crystal $\crtCn$ of type $\tCn$.
The connected components 
$\Gamma_{\fqtpqcm{\crtCn}} (\ew)$,
$\Gamma_{\fqtpqcm{\crtCn}} \parens[\big]{ 1\wbar{1} }$,
$\Gamma_{\fqtpqcm{\crtCn}} \parens[\big]{ 1\wbar{1}1\wbar{1} }$, 
and $\Gamma_{\fqtpqcm{\crtCn}} \parens[\big]{ \wbar{1}11\wbar{1} }$
are respectively
\[
\begin{tikzpicture}[widecrystal,baseline=(b111b1.base)]
  %
  \node (ew) at (0, 0) {\ew};
  \node (1b1) at (1.5, 0) {1\wbar{1}};
  \node (1b11b1) at (3, 0) {1\wbar{1}1\wbar{1}};
  \node (b111b1) at (4.5, 0) {\wbar{1}11\wbar{1}};
  \path (1b1) edge [loop above] node {1} ()
        (1b11b1) edge [loop above] node {1} ()
        (b111b1) edge [loop above] node {1} ();
\end{tikzpicture}
,
\]
which correspond to isolated vertices in $\Gamma_{\fqtpqcm{\crtCn}}$.
We get that $1\wbar{1} \hyco 1\wbar{1}1\wbar{1} \hyco \wbar{1}11\wbar{1}$. 
Although $\wt \parens[\big]{ 1\wbar{1} } = 0 = \wt(\ew)$, 
we have that $1\wbar{1} \nhyco \ew$, because $\ew$ does not have a $1$-labelled loop, while $1\wbar{1}$ has such a loop, that is, 
$\qKoec_1 (\ew) = 0 \neq +\infty = \qKoec_1 \parens[\big]{ 1\wbar{1} }$.
\end{exa}

\section{\texorpdfstring%
  {From $\otimes$-quasi-crystal monoids to $\dotimes$-quasi-crystal monoids}%
  {From tensor product quasi-crystal monoids to quasi-tensor product quasi-crystal monoids}}
\label{sec:tpqcmqtpqcm}

In this section, we investigate wether a $\otimes$-quasi-crystal monoid gives rise to a $\dotimes$-quasi-crystal monoid.
For this purpose, we consider the following monoid property.

\begin{dfn}
\label{dfn:equidivisible}
A monoid $M$ is said to be \dtgterm{equidivisible} if for any elements $x_1, x_2, y_1, y_2 \in M$ satisfying $x_1 y_1 = x_2 y_2$, there exists $z \in M$ such that
$x_2 = x_1 z$ and $y_1 = z y_2$,
or such that $x_1 = x_2 z$ and $y_2 = z y_1$.
\end{dfn}

Equidivisible semigroups were introduced by Levi\avoidcitebreak \cite{Lev44}, and an in-depth study of these semigroups was done by McKnight and Storey in\avoidcitebreak \cite{MS69}.
As examples of equidivisible semigroups, we have free semigroups and completely simple semigroups.

\begin{lem}
\label{lem:tpqcmqtpqcmKasdecomp}
Let $\qcrstM$ be a $\otimes$-quasi-crystal monoid,
and let $x, y \in M$ and $i \in I$ be such that $\qKofc_i (x) > 0$ and $\qKoec_i (y) > 0$.
If $\qKoe_i$ is defined on $xy$, then $\qKoe_i (xy) = x'y'$ for some $x', y' \in M$ such that $\qKofc_i (x') > 0$ and $\qKoec_i (y') > 0$.
If $\qKof_i$ is defined on $xy$, then $\qKof_i (xy) = x'y'$ for some $x', y' \in M$ such that $\qKofc_i (x') > 0$ and $\qKoec_i (y') > 0$.
\end{lem}

\begin{proof}
Suppose that $\qKof_i$ is defined on $xy$.
If $\qKofc_i (x) > \qKoec_i (y)$, then $\qKof_i (xy) = \qKof_i (x) y$,
where $\qKofc_i \parens[\big]{ \qKof_i (x) } = \qKofc_i (x) - 1 \geq \qKoec_i (y) > 0$.
Otherwise, $\qKofc_i (x) \leq \qKoec_i (y)$ and $\qKof_i (xy) = x \qKof_i (y)$,
where $\qKoec_i \parens[\big]{ \qKof_i (y) } = \qKoec_i (y) + 1 > \qKoec_i (y) > 0$.

The result for $\qKoe_i$ follows analogously.
\end{proof}

\begin{lem}
\label{lem:tpqcmqtpqcminvdecomp}
Let $\qcrstM$ be a $\otimes$-quasi-crystal monoid which is equidivisible,
and let $x, y, x', y' \in M$ be such that $x y = x' y'$.
For $i \in I$, if $\qKofc_i (x) > 0$ and $\qKoec_i (y) > 0$, then one of the following conditions holds:
\begin{enumerate}
\item\label{lem:tpqcmqtpqcminvdecompf}
$x' = x'_1 x'_2$ for some $x'_1, x'_2 \in M$ such that $\qKofc_i (x'_1) > 0$ and $\qKoec_i (x'_2) > 0$;

\item\label{lem:tpqcmqtpqcminvdecomps}
$y' = y'_1 y'_2$ for some $y'_1, y'_2 \in M$ such that $\qKofc_i (y'_1) > 0$ and $\qKoec_i (y'_2) > 0$;

\item\label{lem:tpqcmqtpqcminvdecompb}
$\qKofc_i (x') > 0$ and $\qKoec_i (y') > 0$.
\end{enumerate}
\end{lem}

\begin{proof}
Take $i \in I$ such that $\qKofc_i (x) > 0$ and $\qKoec_i (y) > 0$.
As $M$ is equidivisible, there exists $z \in M$ such that one of the following cases hold.
\begin{itemize}
\item Case 1: $x' = x z$ and $y = z y'$.
If $\qKoec_i (z) > 0$, then condition\avoidrefbreak \itmref{lem:tpqcmqtpqcminvdecompf} is verified.
Thus, assume $\qKoec_i (z) = 0$.
This implies that $\qKofc_i (x') \geq \qKofc_i (x) > 0$.
Also, $\qKoec_i (y') = \qKoec_i (z) + \qKoec_i (y') \geq \qKoec_i (y) > 0$.
Hence, condition\avoidrefbreak \itmref{lem:tpqcmqtpqcminvdecompb} is satisfied.

\item Case 2: $x = x' z$ and $y' = z y$.
If $\qKofc_i (z) > 0$, then condition\avoidrefbreak \itmref{lem:tpqcmqtpqcminvdecomps} is verified.
Thus, assume $\qKofc_i (z) = 0$.
This implies that $\qKoec_i (y') \geq \qKoec_i (y) > 0$.
Also, $\qKofc_i (x') = \qKofc_i (x') + \qKofc_i (z) \geq \qKofc_i (x) > 0$.
Hence, condition\avoidrefbreak \itmref{lem:tpqcmqtpqcminvdecompb} is satisfied.
\end{itemize}
Therefore, one condition holds as required.
\end{proof}

\begin{prop}
\label{prop:tpqcmqtpqcm}
Let $\qcrstM_{\otimes} = \parens[\big]{ M; {\cdot}, {\wt}, (\tpqKoe_i)_{i \in I}, (\tpqKof_i)_{i \in I}, (\tpqKoec_i)_{i \in I}, (\tpqKofc_i)_{i \in I} }$ be a $\otimes$-quasi-crystal monoid which is equidivisible.
For each $i \in I$, let $\qtpqKoe_i, \qtpqKof_i : M \to M \sqcup \set{\undf}$ and $\qtpqKoec_i, \qtpqKofc_i : M \to \Z \cup \set{ +\infty }$ be defined for each $x \in M$ by
\begin{enumerate}
\item\label{prop:tpqcmqtpqcmif}
if $x = x_1 x_2$, for some $x_1, x_2 \in M$ satisfying $\tpqKofc_i (x_1) > 0$ and $\tpqKoec_i (x_2) > 0$, set
$\qtpqKoe_i (x) = \qtpqKof_i (x) = \undf$
and
$\qtpqKoec_i (x) = \qtpqKofc_i (x) = +\infty$;

\item\label{prop:tpqcmqtpqcmot}
otherwise, set
$\qtpqKoe_i (x) = \tpqKoe_i (x)$,
$\qtpqKof_i (x) = \tpqKof_i (x)$,
$\qtpqKoec_i (x) = \tpqKoec_i (x)$,
and
$\qtpqKofc_i (x) = \tpqKofc_i (x)$.
\end{enumerate}
Then, $\qcrstM_{\dotimes} = \parens[\big]{ M; {\cdot}, {\wt}, (\qtpqKoe_i)_{i \in I}, (\qtpqKof_i)_{i \in I}, (\qtpqKoec_i)_{i \in I}, (\qtpqKofc_i)_{i \in I} }$ is a $\dotimes$-quasi-\linebreak{}crystal monoid.
\end{prop}

\begin{proof}
Let $x \in M$ and $i \in I$.
If $x = x_1 x_2$, for some $x_1, x_2 \in M$ such that $\tpqKofc_i (x_1) > 0$ and $\tpqKoec_i (x_2) > 0$,
then all axioms of \comboref{Definition}{dfn:qc} are satisfied.
Otherwise, the values of $\qtpqKoe_i$, $\qtpqKof_i$, $\qtpqKoec_i$, and $\qtpqKofc_i$ on $x$ coincide with the values of $\tpqKoe_i$, $\tpqKof_i$, $\tpqKoec_i$, and $\tpqKofc_i$ on $x$, respectively.
Moreover, if $\qtpqKoe_i$ is defined on $x$, then $\qtpqKoe_i (x)$ does not admit a decomposition of the form $x_1 x_2$ with $x_1, x_2 \in M$ such that $\tpqKofc_i (x_1) > 0$ and $\tpqKoec_i (x_2) > 0$,
because $x = \tpqKof_i \parens[\big]{ \tpqKoe_i (x) }$ and, by \comboref{Lemma}{lem:tpqcmqtpqcmKasdecomp}, $x$ would also have such a decomposition.
Analogously, if $\qtpqKof_i$ is defined on $x$, then $\qtpqKof_i (x)$ does not admit a decomposition of the form $x_1 x_2$ with $x_1, x_2 \in M$ such that $\tpqKofc_i (x_1) > 0$ and $\tpqKoec_i (x_2) > 0$.
Then, by iterating this process, for any $k \geq 0$, we have that $(\qtpqKoe_i)^k$ or $(\qtpqKof_i)^k$ is defined on $x$ if and only if $(\tpqKoe_i)^k$ or $(\tpqKof_i)^k$ is defined on $x$, respectively.
As 
$\qcrstM_{\otimes}$ 
is a seminormal quasi-crystal,
it is now immediate that
$\qcrstM_{\dotimes}$ 
is also a seminormal quasi-crystal.

It remains to show that the map $\psi : M \dotimes M \to M$, given by $\psi (x \dotimes y) = xy$ for each $x, y \in M$, is a quasi-crystal homomorphism from $\qcrstM_{\dotimes} \dotimes \qcrstM_{\dotimes}$ to $\qcrstM_{\dotimes}$.

Let $x, y \in M$.
We have that
\[ \wt (x \dotimes y) = \wt(x) + \wt(y) = \wt(xy) = \wt \parens[\big]{ \psi (x \dotimes y) }. \]
Let $i \in I$.
If $x$, $y$, and $xy$ do not admit decompositions of the form $z_1 z_2$ with $z_1, z_2 \in M$ such that $\tpqKofc_i (z_1) > 0$ and $\tpqKoec_i (z_2) > 0$ (in particular, $\tpqKofc_i (x) = 0$ or $\tpqKoec_i (y) = 0$), 
then the values of $\qtpqKoe_i$, $\qtpqKof_i$, $\qtpqKoec_i$, and $\qtpqKofc_i$ on $x$, $y$, and $xy$ coincide with the values of $\tpqKoe_i$, $\tpqKof_i$, $\tpqKoec_i$, and $\tpqKofc_i$ on $x$, $y$, and $xy$, respectively; 
and since $\qcrstM_{\otimes}$ is a $\otimes$-quasi-crystal monoid, we obtain from \comboref{Remark}{rmk:qctpqtp} that $\psi$ satisfies the axioms of \comboref{Definition}{dfn:qch}.
Otherwise, we get one of the following cases.
\begin{itemize}
\item Case 1: $x = x_1 x_2$ for some $x_1, x_2 \in M$ such that $\tpqKofc_i (x_1) > 0$ and $\tpqKoec_1 (x_2) > 0$.
Then, $\qtpqKoec_i (x) = \qtpqKofc_i (x) = +\infty$, which implies that $\qtpqKoec_i (x \dotimes y) = \qtpqKofc_i (x \dotimes y) = +\infty$.
Since $\tpqKoec_i (x_2 y) \geq \tpqKoec_i (x_2) > 0$, we also have that $\qtpqKoec_i (xy) = \qtpqKofc_i (xy) = +\infty$.

\item Case 2: $y = y_1 y_2$ for some $y_1, y_2 \in M$ such that $\tpqKofc_i (y_1) > 0$ and $\tpqKoec_i (y_2) > 0$.
Then, $\qtpqKoec_i (y) = \qtpqKofc_i (y) = +\infty$, which implies that $\qtpqKoec_i (x \dotimes y) = \qtpqKofc_i (x \dotimes y) = +\infty$.
Since $\tpqKofc_i (x y_1) \geq \tpqKofc_i (y_1) > 0$, we also have that $\qtpqKoec_i (xy) = \qtpqKofc_i (xy) = +\infty$.

\item Case 3: $xy = z_1 z_2$ for some $z_1, z_2 \in M$ such that $\tpqKofc_i (z_1) > 0$ and $\tpqKoec_i (z_2) > 0$.
By \comboref{Lemma}{lem:tpqcmqtpqcminvdecomp}, we have case 1, case 2, or $\tpqKofc_i (x) > 0$ and $\tpqKoec_i (y) > 0$.
Then, $\qtpqKoec_i (x \dotimes y) = \qtpqKofc_i (x \dotimes y) = \qtpqKoec_i (xy) = \qtpqKofc_i (xy) = +\infty$.
\end{itemize}
Therefore, $\qcrstM_{\dotimes}$ is a $\dotimes$-quasi-crystal monoid.
\end{proof}

Let $\qcrstQ$ be a seminormal quasi-crystal.
From the previous result, we can easily obtain the quasi-crystal graph $\Gamma_{\fqtpqcm{\qcrstQ}}$ of the free $\dotimes$-quasi-crystal monoid $\fqtpqcm{\qcrstQ}$ from the quasi-crystal graph $\Gamma_{\ftpqcm{\qcrstQ}}$ of the free $\otimes$-quasi-crystal monoid $\ftpqcm{\qcrstQ}$.
Start with $\Gamma_{\ftpqcm{\qcrstQ}}$; 
then, for each word $w \in Q^*$ and each $i \in I$ such that $w = u v$, for some $u, v \in Q^*$ where $u$ is the start of an $i$-labelled edge and $v$ is the end of an $i$-labelled edge, remove any $i$-labelled edge starting or ending on $w$, and then add an $i$-labelled loop on $w$; 
thus, the resulting graph is $\Gamma_{\fqtpqcm{\qcrstQ}}$.

From the following observation, we have an alternative description of the decomposition of an element $x$ as $x_1 x_2$ with $\tpqKofc_i (x_1) > 0$ and $\tpqKoec_i (x_2) > 0$.

\begin{rmk}
Let $\qcrstM$ be a $\otimes$-quasi-crystal monoid, and let $x \in M$.
If $x = x_1 x_2 x_3 x_4 x_5$ with $x_1, x_2, x_3, x_4, x_5 \in M$ such that $\qKofc_i (x_2) > 0$ and $\qKoec_i (x_4) > 0$.
If $\qKofc_i (x_3) > 0$, take $y = x_1 x_2 x_3$ and $z = x_4 x_5$; otherwise, take $y = x_1 x_2$ and $z = x_3 x_4 x_5$.
Then, $x = yz$ where $\qKofc_i (y) > 0$ and $\qKoec_i (z) > 0$.
\end{rmk}

Let $\qcrstQ$ be a seminormal quasi-crystal.
From the previous remark, we have that the quasi-crystal graph $\Gamma_{\fqtpqcm{\qcrstQ}}$ of the free $\dotimes$-quasi-crystal monoid $\fqtpqcm{\qcrstQ}$ can be constructed from the quasi-crystal graph $\Gamma_{\ftpqcm{\qcrstQ}}$ of the free $\otimes$-quasi-crystal monoid $\ftpqcm{\qcrstQ}$ as follows.
Start with $\Gamma_{\ftpqcm{\qcrstQ}}$; 
then, for each word $w \in Q^*$ and each $i \in I$ such that $w = u_1 x u_2 y u_3$, for some $u_1, u_2, u_3 \in Q^*$ and $x, y \in Q$ where $x$ is the start of an $i$-labelled edge and $y$ is the end of an $i$-labelled edge, remove any $i$-labelled edge starting or ending on $w$, and then add an $i$-labelled loop on $w$; 
thus, the resulting graph is $\Gamma_{\fqtpqcm{\qcrstQ}}$.

When $\qcrstQ$ is the standard crystal $\crtAn$ of type $\tAn$,
this construction of $\Gamma_{\fqtpqcm{\crtAn}}$ from $\Gamma_{\ftpqcm{\crtAn}}$ compares with the one given by Cain and Malheiro in\avoidcitebreak \cite{CM17crysthypo}.

\begin{exa}
\label{exa:ftpqcmA3fqtpqcmA3}
Consider the standard crystal $\crtA_3$ of type $\tA_3$.
From the connected component $\Gamma_{\ftpqcm{\crtA_3}} (112)$, described in \comboref{Example}{exa:plA3}, 
we remove the $1$-labelled edge $112 \lbedge{1} 212$; 
add $1$-labelled loops on $112$, $212$, and $312$; 
remove the $2$-labelled edge $223 \lbedge{2} 323$; 
and adde $2$-labelled loops on $213$, $223$, and $323$.
The resulting quasi-crystal graph is
\[
\begin{tikzpicture}[widecrystal,baseline=(223.base)]
  %
  \node (112) at (1, 1.5) {112};
  \node (212) at (2, 2) {212};
  \node (312) at (3, 2) {312};
  \node (313) at (4, 2) {313};
  \node (113) at (2, 1) {113};
  \node (213) at (3, 1) {213};
  \node (223) at (4, 1) {223};
  \node (323) at (5, 1.5) {323};
  \path (212) edge node {2} (312)
        (312) edge node {2} (313)
        (313) edge node {1} (323)
        (112) edge[swap] node {2} (113)
        (113) edge[swap] node {1} (213)
        (213) edge[swap] node {1} (223)
        (112) edge [loop above] node {1} ()
        (212) edge [loop above] node {1} ()
        (312) edge [loop above] node {1} ()
        (213) edge [loop below] node {2} ()
        (223) edge [loop below] node {2} ()
        (323) edge [loop below] node {2} ();
\end{tikzpicture}
\]
which cooresponds to the connected components $\Gamma_{\fqtpqcm{\crtA_3}} (112)$ and $\Gamma_{\fqtpqcm{\crtA_3}} (212)$.
\end{exa}

\begin{exa}
\label{exa:ftpqcmC3fqtpqcmC3}
Consider the standard crystal $\crtC_3$ of type $\tC_3$.
From the connected component $\Gamma_{\ftpqcm{\crtC_3}} \parens[\big]{ 12\wbar{2} }$, described in \comboref{Example}{exa:plC3}, 
we remove the $1$-labelled edges $12\wbar{2} \lbedge{1} 12\wbar{1}$ and $1\wbar{2}\,\wbar{1} \lbedge{1} 2\wbar{2}\,\wbar{1}$; 
add $1$-labelled loops on each vertex; 
and adde $2$-labelled loops on $12\wbar{2}$ and $2\wbar{2}\,\wbar{1}$.
The resulting quasi-crystal graph is
\[
\begin{tikzpicture}[widecrystal,baseline=(223.base)]
  %
  \node (12b2) at (0, 0) {12\wbar{2}};
  \node (12b1) at (1, 0) {12\wbar{1}};
  \node (13b1) at (2, 0) {13\wbar{1}};
  \node (1b3b1) at (3, 0) {1\wbar{3}\,\wbar{1}};
  \node (1b2b1) at (4, 0) {1\wbar{2}\,\wbar{1}};
  \node (2b2b1) at (5, 0) {2\wbar{2}\,\wbar{1}};
  \path (12b1) edge node {2} (13b1)
        (13b1) edge node {3} (1b3b1)
        (1b3b1) edge node {2} (1b2b1)
        (12b2) edge [loop above] node {1} ()
        (12b1) edge [loop above] node {1} ()
        (13b1) edge [loop above] node {1} ()
        (1b3b1) edge [loop above] node {1} ()
        (1b2b1) edge [loop above] node {1} ()
        (2b2b1) edge [loop above] node {1} ()
        (12b2) edge [loop left] node {2} ()
        (2b2b1) edge [loop right] node {2} ();
\end{tikzpicture}
\]
which cooresponds to the connected components 
$\Gamma_{\fqtpqcm{\crtC_3}} \parens[\big]{ 12\wbar{2} }$,
$\Gamma_{\fqtpqcm{\crtC_3}} \parens[\big]{ 12\wbar{1} }$,
and $\Gamma_{\fqtpqcm{\crtC_3}} \parens[\big]{ 2\wbar{2}\,\wbar{1} }$.
\end{exa}

\section{Hypoplactic monoids as quotients of plactic monoids}
\label{sec:hypomplacm}

From the presentation for the classical hypoplactic monoid given by Krob and Thibon\avoidcitebreak \cite{KT97}, which contains the Knuth relations\avoidcitebreak \cite{Knu70}, 
it is immediate that the hypoplactic monoid $\hypo(\crtAn)$ of type $\tAn$ is a quotient of the plactic monoid $\plac (\crtAn)$ of type $\tAn$.
On the other hand, as shown in\avoidcitebreak \cite{CGM23quasi-crystals}, the hypoplactic monoid $\hypo(\crtCn)$ of type $\tCn$ contains submonoids isomorphic to $\hypo(\crtA_{n-1})$ and $\hypo(\crtC_{n-1})$, 
but it is not a quotient of the plactic monoid $\plac(\crtCn)$ of type $\tCn$.
Understanding the properties that originate these results is an active research problem.

From the following results, we get a sufficient condition for a hypoplactic monoid to be a quotient of the plactic monoid associated to the same seminormal quasi-crystal.

\begin{lem}
\label{lem:plcoliinv}
Consider a seminormal quasi-crystal $\qcrstQ$. 
Let $x_1, \ldots, x_m \in Q$, and let $\sigma$ be a permutation of $\set{1,\ldots,m}$ such that $x_1 \ldots x_m \plco x_{\sigma(1)} \ldots x_{\sigma(m)}$. 
Then, for each $i \in I$, there exist $k, l \in \set{1,\ldots,m}$ such that $k < l$, $\qKofc_i (x_k) > 0$ and $\qKoec_i (x_l) > 0$
if and only if
there exist $k', l' \in \set{1,\ldots,m}$ such that $\sigma(k') < \sigma(l')$, $\qKofc_i (x_{\sigma(k')}) > 0$, and $\qKoec_i (x_{\sigma(l')}) > 0$.
\end{lem}

\begin{proof}
Set $u = x_1 \ldots x_m$ and $v = x_{\sigma(1)} \ldots x_{\sigma(m)}$.
Let $i \in I$.
By way of contradiction, suppose that there exist $k < l$ such that $\qKofc_i (x_k) > 0$ and $\qKoec_i (x_l) > 0$, and that $\qKofc_i (x_{\sigma(k')}) > 0$ and $\qKoec_i (x_{\sigma(l')}) > 0$ implies $\sigma(k') \geq \sigma(l')$.
Then, in the expression $\tpsgn_i (x_1) \cdots \tpsgn_i (x_m)$, we have at least a ${+}$ to the left of a ${-}$, 
while in the expression $\tpsgn_i (x_{\sigma(1)}) \cdots \tpsgn_i (x_{\sigma(m)})$, every ${-}$ is to the left of the first ${+}$.
Hence, in $\ftpqcm{\qcrstQ}$, we have that
\[
  \qKoec_i (u) < \qKoec_i (x_1) + \cdots + \qKoec_i (x_m) = \qKoec_i (x_{\sigma(1)}) + \cdots + \qKoec_i (x_{\sigma(m)}) = \qKoec_i (v),
\]
which is a contradiction, because $u \plco v$.

The converse implication follows analogously.
\end{proof}

\begin{prop}
Let $\qcrstQ$ be a seminormal quasi-crystal such that
$\set[\big]{ \wt(x) \given x \in Q }$
is a linearly independent set.
Then, ${\plco} \subseteq {\hyco}$.
Therefore, $\hypo(\qcrstQ)$ is a monoid quotient of $\plac(\qcrstQ)$.
\end{prop}

\begin{proof}
Let $u, v \in Q^*$ be such that $u \plco v$.
From \comboref{Definitions}{dfn:fqcmtp} and\avoidrefbreak \ref{dfn:fqcmqtp}, the weight maps of $\ftpqcm{\qcrstQ}$ and $\fqtpqcm{\qcrstQ}$ coincide.
We write a superscript $\otimes$ on the other quasi-crystal structure maps of $\ftpqcm{\qcrstQ}$ and write a superscript $\dotimes$ on the other quasi-crystal structure maps of $\fqtpqcm{\qcrstQ}$.
From the definitions, for any $w \in Q^*$ and $i \in I$, we have that if $\qtpqKoe_i$ (or $\qtpqKof_i$) is defined on $w$, then $\qtpqKoe_i (w) = \tpqKoe_i (w)$ (resp., $\qtpqKof_i (w) = \tpqKof_i (w)$).

Denote the underlying sets of $\ftpqcm{\qcrstQ} (u)$, $\ftpqcm{\qcrstQ} (v)$, $\fqtpqcm{\qcrstQ} (u)$, and $\fqtpqcm{\qcrstQ} (v)$ by $Q^*_{\otimes} (u)$, $Q^*_{\otimes} (v)$, $Q^*_{\dotimes} (u)$, and $Q^*_{\dotimes} (v)$, respectively.
Clearly, $Q^*_{\dotimes} (u) \subseteq Q^*_{\otimes} (u)$ and $Q^*_{\dotimes} (v) \subseteq Q^*_{\otimes} (v)$.
As $u \plco v$, consider a quasi-crystal isomorphism $\psi_{\otimes} : \ftpqcm{\qcrstQ} (u) \to \ftpqcm{\qcrstQ} (v)$ such that $\psi(u) = v$.
Let $\psi_{\dotimes}$ denote the restriction of $\psi_{\otimes}$ to $Q^*_{\dotimes} (u)$. 
Thus, $\psi_{\dotimes} (u) = v$.
We now show that $\psi_{\dotimes}$ is a quasi-crystal isomorphism between $\fqtpqcm{\qcrstQ} (u)$ and $\fqtpqcm{\qcrstQ} (v)$.

Let $u' \in Q^*_{\dotimes} (u)$, and set $v' = \psi_{\dotimes} (u')$.
We have that
\[ \wt(v') = \wt \parens[\big]{ \psi_{\dotimes} (u') } = \wt \parens[\big]{ \psi_{\otimes} (u') } = \wt(u'). \]
Let $i \in I$.
If $u'$ and $v'$ do not admit a decomposition of the form $w_1 x w_2 y w_3$ with $w_1, w_2, w_3 \in Q^*$ and $x, y \in Q$ such that $\qtpqKofc_i (x) > 0$ and $\qtpqKoec_i (y) > 0$,
then the values of $\qtpqKoe_i$, $\qtpqKof_i$, $\qtpqKoec_i$, and $\qtpqKofc_i$ on $u'$ and $v'$ coincide with the values of $\tpqKoe_i$, $\tpqKof_i$, $\tpqKoec_i$, and $\tpqKofc_i$ on $u'$ and $v'$, respectively,
which implies that 
$\psi_{\dotimes} \parens[\big]{ \qtpqKoe_i (u') } = \qtpqKoe_i \parens[\big]{ \psi_{\dotimes} (u') }$,
$\psi_{\dotimes} \parens[\big]{ \qtpqKof_i (u') } = \qtpqKof_i \parens[\big]{ \psi_{\dotimes} (u') }$,
$\qtpqKoec_i \parens[\big]{ \psi_{\dotimes} (u') } = \qtpqKoec_i (u')$,
and $\qtpqKofc_i \parens[\big]{ \psi_{\dotimes} (u') } = \qtpqKofc_i (u')$,
because $u' \plco v'$ as $\psi_{\otimes} (u') = v'$.

Otherwise, $u'$ or $v'$ admits a decomposition of the form $w_1 x w_2 y w_3$ with $w_1, w_2, w_3 \in Q^*$ and $x, y \in Q$ such that $\qtpqKofc_i (x) > 0$ and $\qtpqKoec_i (y) > 0$.
Take $x_1, \ldots, x_p, y_1, \ldots, y_q \in Q$ with $u' = x_1 \ldots x_p$ and $v' = y_1 \ldots y_q$.
Thus, there exist $k > l$ such that $\qtpqKofc_i (x_k) > 0$ and $\qtpqKoec_i (x_l) > 0$, 
or there exist $k' > l'$ such that $\qtpqKofc_i (y_{k'}) > 0$ and $\qtpqKoec_i (y_{l'}) > 0$.
Since the set
$\set[\big]{ \wt(x) \given x \in Q }$
is linearly independent and
\[ \wt(x_1) + \cdots + \wt(x_p) = \wt(u') = \wt(v') = \wt(y_1) + \cdots + \wt(y_q), \]
we get that $p=q$ and the letters $x_1, \ldots, x_p$ and $y_1, \ldots, y_p$ are the same and occur exactly the same number of times.
That is, there exists a permutation $\sigma$ on $\set{1,\ldots,m}$ such that $v' = x_{\sigma(1)} \ldots x_{\sigma(p)}$.
Note that $\qtpqKoec_i$ and $\qtpqKofc_i$ coincide with $\tpqKoec_i$ and $\tpqKofc_i$ on $Q$, respectively.
As $u' \plco v'$, by \comboref{Lemma}{lem:plcoliinv}, there exist $k, l, k', l' \in \set{1,\ldots,p}$ such that $k < l$, $\qtpqKofc_i (x_k) > 0$, $\qtpqKoec_i (x_l) > 0$, $\sigma(k') > \sigma(l')$, $\qtpqKofc_i (x_{\sigma(k')}) > 0$, and $\qtpqKoec_i (x_{\sigma(l')}) > 0$.
Hence, $\qtpqKoec_i \parens[\big]{ \psi_{\dotimes} (u') } = \qtpqKofc_i \parens[\big]{ \psi_{\dotimes} (u') } = \qtpqKoec_i (u') = \qtpqKofc_i (u') = +\infty$.

We get that $\psi_{\dotimes}$ is a quasi-crystal homomorphism from $\fqtpqcm{\qcrstQ} (u)$ to $\fqtpqcm{\qcrstQ} (v)$.
Analogously, as $(\psi_{\otimes})^{-1}$ is a quasi-crystal isomorphism between $\ftpqcm{\qcrstQ} (v)$ and $\ftpqcm{\qcrstQ} (u)$, we obtain by a similar reasoning that its restriction to $Q^*_{\dotimes} (v)$ is a quasi-crystal homomorphism from $\fqtpqcm{\qcrstQ} (v)$ to $\fqtpqcm{\qcrstQ} (u)$.
Therefore, $\psi_{\dotimes}$ is a quasi-crystal isomorphism between $\fqtpqcm{\qcrstQ} (u)$ and $\fqtpqcm{\qcrstQ} (v)$, which implies $u \hyco v$.
\end{proof}

For the standard crystal $\crtAn$ of type $\tAn$, the weights $\wt(x) = \vc{e_x}$, $x \in A_n$, are linearly independent.
Thus, from the previous proposition, we get the well-known result that the classical hypoplactic monoid is a quotient of the classical plactic monoid.

\bibliography{\jobname}

\end{document}